\documentclass[10pt,a4paper]{article}

\usepackage[T1]{fontenc}
\usepackage{lmodern}
\usepackage{amsfonts,amsmath,amssymb,amsthm,mathtools}
\usepackage{enumitem}
\usepackage{microtype}
\usepackage{titlesec}
\usepackage[hidelinks]{hyperref}
\usepackage[a4paper,top=1.7cm,bottom=1.7cm,left=1.8cm,right=1.8cm]{geometry}

\hypersetup{
  pdftitle={The semi-inducibility of the blue-blue-red path on four vertices},
  pdfauthor={Jinghua Deng}
}

\allowdisplaybreaks[1]
\titlespacing*{\section}{0pt}{2.2ex plus .8ex minus .2ex}{1.0ex plus .2ex}
\titlespacing*{\subsection}{0pt}{1.8ex plus .6ex minus .2ex}{.8ex plus .2ex}
\AtBeginDocument{%
  \setlength{\abovedisplayskip}{7pt plus 2pt minus 3pt}%
  \setlength{\belowdisplayskip}{7pt plus 2pt minus 3pt}%
  \setlength{\abovedisplayshortskip}{4pt plus 2pt}%
  \setlength{\belowdisplayshortskip}{5pt plus 2pt minus 2pt}%
  \setlength{\jot}{2pt}%
}

\newtheorem{definition}{Definition}[section]
\newtheorem{theorem}[definition]{Theorem}
\newtheorem{lemma}[definition]{Lemma}
\newtheorem*{fubinitheorem}{Fubini's theorem}

\newcommand{\one}{\mathbf 1}
\newcommand{\Ecal}{\mathcal E}
\newcommand{\Hcal}{\mathcal H}
\newcommand{\gfun}{\psi}
\newcommand{\gcvx}{\underline\psi}
\newcommand{\Freg}{F}
\newcommand{\pstar}{p_*}
\newcommand{\Dcal}{\mathcal D}
\newcommand{\dd}{\,\mathrm d}
\newcommand{\ip}[2]{\left\langle #1,#2\right\rangle}
\newcommand{\LSrandomreference}{%
  \cite[Theorem~2.5 and Corollary~2.6]{LovaszSzegedy}}

\title{\bf\Large The semi-inducibility of the blue--blue--red path on four vertices}
\author{Jinghua Deng\thanks{Email: Jinghua\_deng@163.com}\\
\small School of Data Science and Engineering, South China Normal University}
\date{July 11, 2026}

\begin{document}
\maketitle

\begin{abstract}
For an $n$-vertex graph $G$, let $N(H_3,G)$ be the number of injective labeled
copies of the red-blue path $H_3$ for which the two blue pairs are mapped to
non-edges of $G$ and the red pair is mapped to an edge of $G$. We determine
the maximum limiting value of $N(H_3,G)/n^4$ and give an extremal construction,
which is the disjoint union of a clique and an asymptotically regular graph.
The proof uses weighted vertex quotients and degree-square tie-breaking. We
thereby resolve the exceptional four-vertex case left open in the recent
classification of non-complete red-blue graphs.
\end{abstract}

\section{Introduction}\label{sec:introduction}

The inducibility problem, introduced by Pippenger and
Golumbic~\cite{PippengerGolumbic}, asks how large the asymptotic density of
induced copies of a fixed graph can be. Exact values are known for several
families and individual small graphs; see, for example,
\cite{BrownSidorenko,BollobasEgawaHarrisJin,HatamiHirstNorine,
BodnarPikhurkoFlag,BodnarGaoLeonLiuPikhurkoSun}. Nevertheless, even
four-vertex examples can require substantial work. The principal approaches
include graph-limit compactness~\cite{LovaszSzegedy,Lovasz}, which turns the
problem into an optimization over graphons, the flag algebra method of
Razborov~\cite{Razborov}, and symmetrization and stability arguments for
identifying extremal constructions.

The semi-inducibility problem, introduced by Basit, Granet, Horsley,
K\"undgen and Staden~\cite{BasitEtAl}, is a red-blue extension of inducibility.
A \emph{red-blue graph} $H$ is a graph whose edges are colored red or blue;
pairs which are not edges of $H$ are left unconstrained. We view an $n$-vertex
host graph $G$ as a complete red-blue graph by coloring its edges red and its
non-edges blue. A semi-induced copy of $H$ in $G$ is an injective map
$\phi:V(H)\to V(G)$ such that
\begin{align*}
 \phi(x)\phi(y)&\in E(G)\quad\text{whenever $xy$ is red},&
 \phi(x)\phi(y)&\notin E(G)\quad\text{whenever $xy$ is blue}.
\end{align*}
Writing $|V(H)|=h$, define
\begin{align*}
 N(H,G)&\coloneqq
 \#\{\text{semi-induced injective labeled copies of $H$ in $G$}\},&
 p(H,G)&\coloneqq\frac{N(H,G)}{n^h}.
\end{align*}
The semi-inducibility problem asks for the largest limiting value of
$p(H,G)$ over all $n$-vertex host graphs. If every pair of $H$ is prescribed,
this is the usual inducibility problem; if only red pairs are prescribed, it
reduces to ordinary subgraph density.

Basit et al.~\cite{BasitEtAl} obtained sharp or asymptotically sharp results
for alternating walks, alternating cycles of length divisible by four, and
four-cycles under every color pattern. Chen and Noel~\cite{ChenNoel} later
proved that the alternating six-cycle is asymptotically maximized by the
uniformly random coloring, settling the first open alternating-cycle case.
For four vertices, Bodn\'ar and Pikhurko~\cite{BodnarPikhurko} considered all
non-complete red-blue graphs and resolved every remaining case except the
three-edge path colored blue--blue--red, together with its equivalent forms.

Recent work also treats semi-inducibility under a prescribed host-graph edge
density. Balogh, Lidick\'y, Mubayi, Pfender and Volec~\cite{BaloghEtAl}
obtained sharp profiles for several small red-blue graphs, and Deng and
Hou~\cite{DengHou} determined both profiles of the semi-induced four-vertex
star. For the unrestricted problem, define
\begin{align}\label{eq:I-H3}
 I(H)\coloneqq\lim_{n\to\infty}\max_{|V(G)|=n}p(H,G).
\end{align}
The existence of the limit follows from the standard blow-up argument, or
from compactness of dense graph limits~\cite{BasitEtAl,Lovasz}.

In this paper we determine the exceptional four-vertex case isolated by
Bodn\'ar and Pikhurko~\cite{BodnarPikhurko}. Let $H_3$ be the red-blue graph on
ordered vertices
$v_0,v_1,v_2,v_3$ in which $v_0v_1$ and $v_1v_2$ are blue, $v_2v_3$ is red,
and the remaining three pairs are uncolored.

For the extremal construction, let $1/2\le a\le1$ and
$0\le x\le1-a$. Partition $V(G_n)=A_n\cup B_n$ with
\begin{align*}
 |A_n|=(a+o(1))n,\qquad |B_n|=(1-a+o(1))n.
\end{align*}
Put all edges inside $A_n$, no edges between $A_n$ and $B_n$, and choose
$G_n[B_n]$ to be asymptotically $xn$-regular. The limiting red edge density is
$\beta(a,x)=a^2+(1-a)x$.
The normalized red and blue degrees are $a,1-a$ on $A_n$ and $x,1-x$ on
$B_n$, respectively.
Up to $o(1)$, choosing the first edge blue and the last edge red contributes
$\beta(a,x)(1-\beta(a,x))$. We subtract the choices for which the middle pair
is red. When this pair lies in $A_n$, its contribution is
$a^2(1-a)a=a^3(1-a)$. When it lies in $B_n$, its contribution is
$(1-a)x(1-x)x=(1-a)x^2(1-x)$. Hence the limiting $H_3$-density is
\begin{align}\label{eq:F-construction}
 F(a,x)
  ={}&(a^2+(1-a)x)(1-a^2-(1-a)x)-a^3(1-a)-(1-a)x^2(1-x).
\end{align}

Let $x_*\in(0,1/4)$ be the unique root of
\begin{align}\label{eq:cubic-d}
 32x^3-40x^2+13x-1=0,
\end{align}
and put
\begin{equation}\label{eq:star-parameters}
 a_*\coloneqq\frac{4x_*^2-7x_*+2}{3-8x_*},\qquad
 \lambda_*\coloneqq\frac{x_*}{1-a_*},\qquad
 \pstar\coloneqq F(a_*,x_*).
\end{equation}
\begin{theorem}\label{thm:main}
For the blue--blue--red path $H_3$, with the parameters defined in
\eqref{eq:cubic-d} and \eqref{eq:star-parameters},
\begin{align*}
 I(H_3)=\pstar.
\end{align*}
Moreover, equality is attained by a graph sequence of the form
\begin{align*}
 G_n=K_{(a_*+o(1))n}\sqcup R_n,
\end{align*}
where $|V(R_n)|=(1-a_*+o(1))n$ and every vertex of $R_n$ has degree
$(x_*+o(1))n$, equivalently relative degree $\lambda_*+o(1)$ inside $R_n$.
\end{theorem}

The proof is organized as follows. Section~\ref{sec:weighted-model} gives the
weighted reduction and the degree identity. In
Section~\ref{sec:split-regular} we optimize the proposed construction. The
structural reduction is proved in Section~\ref{sec:structural-reduction}: a
degree-square-minimal maximizer is forced to be the disjoint union of a clique
of weight greater than $1/2$ and a graphon of maximum degree at most $1/2$.
Section~\ref{sec:main-proof} then reduces the upper bound to the two-variable
optimization from Section~\ref{sec:split-regular}.

\section{Preliminary}\label{sec:weighted-model}

\begin{fubinitheorem}[{Folland~\cite[Theorem~2.37]{Folland}}]
Let $p,q\ge 1$, and let $f:[0,1]^{p+q}\to\mathbb R$ be bounded and
measurable. Then $f$ is integrable and
\[
 \int_{[0,1]^{p+q}}f(x,y)\dd(x,y)
 =\int_{[0,1]^p}\left(\int_{[0,1]^q}f(x,y)\dd y\right)\dd x
 =\int_{[0,1]^q}\left(\int_{[0,1]^p}f(x,y)\dd x\right)\dd y.
\]
\end{fubinitheorem}

Let $V$ be a finite vertex set, and assign each vertex $i\in V$ a positive
weight $\alpha_i$, with $\sum_{i\in V}\alpha_i=1$. Let
$W=(W_{ij})_{i,j\in V}$ be a symmetric matrix with entries in $[0,1]$, where
$W_{ij}$ is the red weight assigned to the pair $ij$. Thus
$(V,\alpha,W)$ is a finite weighted red graph. Under a blow-up, the vertex $i$
is replaced by a class of relative size $\alpha_i$, and $W_{ij}$ is realized
as the red edge density between the classes corresponding to $i$ and $j$;
$W_{ii}$ is the red density inside the class corresponding to $i$. Put
\begin{equation*}\tag{4a}\label{eq:weighted-parameters}
 \overline W_{ij}:=1-W_{ij},\qquad
 d_i:=\sum_j\alpha_jW_{ij},\qquad
 r_i:=1-d_i,\qquad
 \beta:=\sum_i\alpha_id_i.
\end{equation*}
We extend the density to this weighted model by choosing
$i_0,i_1,i_2,i_3$ independently from $V$, with
$\Pr(i_k=i)=\alpha_i$. For a fixed ordered choice, the factor
$\overline W_{i_0i_1}\overline W_{i_1i_2}W_{i_2i_3}$ is its contribution to
the event that the first two prescribed pairs are blue and the last one is
red. Thus the weighted $H_3$-density is
\begin{align}
 p(H_3,W)\coloneqq\Phi(W)
 :={}&\sum_{i_0,i_1,i_2,i_3\in V}
 \alpha_{i_0}\alpha_{i_1}\alpha_{i_2}\alpha_{i_3}
 \overline W_{i_0i_1}\overline W_{i_1i_2}W_{i_2i_3}.
 \label{eq:Phi-direct}
\end{align}
The original density $p(H_3,G)$ counts only injective tuples, whereas the
indices in \eqref{eq:Phi-direct} need not be distinct. Their discrepancy for
ordinary graphs is quantified in Lemma~\ref{lem:sequence-bridge}.

Fix the middle ordered pair $i=i_1$ and $j=i_2$. The total weight of
choices for $i_0$ joined to $i$ by a blue edge is
$\sum_{i_0}\alpha_{i_0}\overline W_{i_0i}=r_i$, while the total weight of
choices for $i_3$ joined to $j$ by a red edge is
$\sum_{i_3}\alpha_{i_3}W_{ji_3}=d_j$. Hence
\begin{equation}
 \Phi(W)=\sum_{i,j}\alpha_i\alpha_j\overline W_{ij}r_id_j
 =\sum_{i,j}\alpha_i\alpha_j\overline W_{ij}r_i(1-r_j).
 \label{eq:Phi-middle}
\end{equation}
Since $\sum_j\alpha_j\overline W_{ij}=r_i$, expanding
\eqref{eq:Phi-middle} gives the alternative form
\eqref{eq:Phi-Qform}:
\begin{equation}
 \Phi(W)=\sum_i\alpha_ir_i^2
 -\sum_{i,j}\alpha_i\alpha_j\overline W_{ij}r_ir_j.
 \label{eq:Phi-Qform}
\end{equation}

A graphon is a symmetric measurable function
$W:[0,1]^2\to[0,1]$, where $W(x,y)$ is interpreted as the red weight of the
pair $xy$. To represent the finite weighted graph $(V,\alpha,W)$, partition
$[0,1]$ into intervals $(I_i)_{i\in V}$ with $|I_i|=\alpha_i$, and set
$W(x,y):=W_{ij}$ for $x\in I_i$ and $y\in I_j$.
We use the same symbol $W$ for the weighted matrix and its associated step
graphon. The interval $[0,1]$ has total measure $1$, corresponding to the
normalization $\sum_i\alpha_i=1$. Define
\[
 d_W(x):=\int_0^1W(x,y)\dd y,
 \qquad
 \beta_W:=\int_0^1d_W(x)\dd x
 =\int_{[0,1]^2}W(x,y)\dd x\dd y.
\]
If $x\in I_i$, then
\[
 d_W(x)=\sum_j\alpha_jW_{ij}=d_i,
 \qquad
 \beta_W=\sum_i\alpha_id_i=\beta.
\]
Thus $d_W(x)$ is the normalized red degree of $x$ and $\beta_W$ is the red
edge density. With $\overline W(x,y):=1-W(x,y)$, define
\begin{equation*}\tag{6a}\label{eq:Phi-graphon}
 p(H_3,W)\coloneqq\Phi(W)
 :=\int_{[0,1]^4}
 \overline W(x_0,x_1)\overline W(x_1,x_2)W(x_2,x_3)
 \dd x_0\dd x_1\dd x_2\dd x_3.
\end{equation*}
We write $d=d_W$ and $\beta=\beta_W$ when $W$ is clear.

For an ordinary graph $G$ on $n$ vertices, let $W_G$ be the finite weighted
red graph with vertex set $V(G)$, weights $1/n$, and edge densities equal to
the adjacency matrix of $G$.

For a graphon $U$, let $G(n,U)$ be the random graph obtained by sampling
$z_1,\ldots,z_n$ independently and uniformly from $[0,1]$ and, conditional
on these points, including each edge $ij$ independently with probability
$U(z_i,z_j)$. Write $t(F,\cdot)$ for the ordinary homomorphism density of
$F$.

\begin{lemma}[Lov\'asz and Szegedy~\LSrandomreference]
  \label{lem:W-random-convergence}
For every graphon $U$, almost surely,
\[
 t(F,G(n,U))\longrightarrow t(F,U)
\]
simultaneously for every finite simple graph $F$.
\end{lemma}

\begin{lemma}[Lov\'asz and Szegedy~\cite{LovaszSzegedy} and
  Lov\'asz~\cite{Lovasz}]
  \label{lem:dense-graph-compactness}
Let $(G_n)$ be a sequence of finite graphs with $|V(G_n)|\to\infty$. Then
there are $n_k\to\infty$ and a graphon $U$ such that
\[
 t(F,W_{G_{n_k}})\longrightarrow t(F,U)
\]
simultaneously for every finite simple graph $F$.
\end{lemma}

\begin{lemma}\label{lem:sequence-bridge}
  Uniformly over all $n$-vertex graphs $G$,
  \begin{equation}
  p(H_3,G)=\Phi(W_G)+O\left(\frac1n\right).
  \label{eq:finite-to-weighted}
  \end{equation}
  Moreover,
  \begin{equation}
  I(H_3)
  =\sup_{W\text{ finite weighted}}\Phi(W)
  =\max_{W\text{ graphon}}\Phi(W).
  \label{eq:max-weighted}
  \end{equation}
\end{lemma}

\begin{proof}
For $W_G$, restricting the sum in \eqref{eq:Phi-direct} to pairwise distinct
indices gives exactly $N(H_3,G)/n^4=p(H_3,G)$. Let $R_G$ be the number of
non-injective ordered four-tuples that contribute to this sum. Then
\begin{equation}
 \Phi(W_G)=p(H_3,G)+\frac{R_G}{n^4},
 \label{eq:noninjective-error}
\end{equation}
where $0\le R_G\le n^4-(n)_4=6n^3-11n^2+6n$. Consequently,
\begin{equation}
 0\le \Phi(W_G)-p(H_3,G)
 \le \frac6n-\frac{11}{n^2}+\frac6{n^3}
 =O\left(\frac1n\right)
 \label{eq:bridge-uniform}
\end{equation}
uniformly over $G$. Equations \eqref{eq:noninjective-error} and
\eqref{eq:bridge-uniform} prove \eqref{eq:finite-to-weighted}.

Put $S\coloneqq\sup_{W\text{ finite weighted}}\Phi(W)$.
Since $\Phi(W_G)\le S$, \eqref{eq:finite-to-weighted} gives
\[
 \max_{|V(G)|=n}p(H_3,G)\le S+O(1/n).
\]
Hence \eqref{eq:I-H3} gives $I(H_3)\le S$.

For every graphon $U$, inclusion--exclusion gives
\begin{equation}
 \Phi(U)
 =t(K_2,U)-t(K_2\sqcup K_2,U)-t(P_3,U)+t(P_4,U),
 \label{eq:Phi-inclusion-exclusion}
\end{equation}
where $P_k$ is the path on $k$ vertices. Hence
Lemma~\ref{lem:W-random-convergence}, together with
$t(F,G)=t(F,W_G)$ and \eqref{eq:Phi-inclusion-exclusion}, gives
$\Phi(W_{G(n,U)})\to\Phi(U)$ almost surely.
By \eqref{eq:finite-to-weighted},
\[
 p(H_3,G(n,U))\longrightarrow\Phi(U)
 \qquad\text{almost surely}.
\]
Fixing one realization from this probability-one event and applying
\eqref{eq:I-H3} gives $\Phi(U)\le I(H_3)$. Since every finite weighted graph is
a step graphon,
\[
 I(H_3)\le S
 \le\sup_{U\text{ graphon}}\Phi(U)
 \le I(H_3).
\]
Thus
\[
 I(H_3)=S=\sup_{U\text{ graphon}}\Phi(U).
\]

For each $n$, choose an $n$-vertex graph $G_n$ such that
$p(H_3,G_n)=\max_{|V(G)|=n}p(H_3,G)$.
By Lemma~\ref{lem:dense-graph-compactness}, there are $n_k\to\infty$ and a
graphon $U_*$ such that
$
 t(F,W_{G_{n_k}})\to t(F,U_*)
$
for every finite simple graph $F$. Consequently,
\eqref{eq:Phi-inclusion-exclusion}, \eqref{eq:finite-to-weighted}, and
\eqref{eq:I-H3} yield
\[
 \Phi(U_*)
 =\lim_{k\to\infty}\Phi(W_{G_{n_k}})
 =\lim_{k\to\infty}p(H_3,G_{n_k})
 =I(H_3).
\]
Therefore $I(H_3)=S=\max_{U\text{ graphon}}\Phi(U)$, proving
\eqref{eq:max-weighted}.
\end{proof}

For functions $f,g:V\to\mathbb R$, write
\[
 \ip{f}{g}:=\sum_i\alpha_if_ig_i,
 \qquad \|f\|_2^2:=\ip{f}{f},
 \qquad \|f\|_1:=\sum_i\alpha_i|f_i|.
\]
For $A\subseteq V$, write $|A|:=\sum_{i\in A}\alpha_i$. For a finite
weighted red graph $U$, define
\[
 (Uf)_i:=\sum_j\alpha_jU_{ij}f_j,
 \qquad
 \Ecal_U(f):=\frac12\sum_{i,j}\alpha_i\alpha_j
 U_{ij}(f_i-f_j)^2.
\]
For a graphon $U$, measurable $A\subseteq[0,1]$, and $f,g\in L^2([0,1])$,
the corresponding notation is
\[
 \ip{f}{g}:=\int_0^1f(x)g(x)\dd x,
 \qquad \|f\|_2^2:=\ip{f}{f},
 \qquad \|f\|_1:=\int_0^1|f(x)|\dd x,
\]
\[
 |A|:=\int_A\dd x,
 \qquad (Uf)(x):=\int_0^1U(x,y)f(y)\dd y,
\]
and
\[
 \Ecal_U(f):=\frac12\int_{[0,1]^2}
 U(x,y)(f(x)-f(y))^2\dd x\dd y.
\]
Let
\begin{equation}
 \gfun(t):=t^2(1-t).
 \label{eq:g-def}
\end{equation}

\begin{lemma}\label{lem:dirichlet}
For every finite weighted red graph $W$,
\begin{align}
 \Phi(W)
 &={}\beta(1-\beta)-\sum_i\alpha_i\gfun(d_i)-\Ecal_W(d)
 \label{eq:dirichlet}
 \\
 &={}\beta(1-\beta)-\|d\|_2^2+\ip{d}{Wd}.
 \label{eq:operator-form}
\end{align}
Since $d=W\one$, set $B:=W(\one-d)=d-Wd$. Equivalently,
\begin{equation}
 B_i=\sum_j\alpha_jW_{ij}(1-d_j)
 =d_i-\sum_j\alpha_jW_{ij}d_j.
 \label{eq:B-def}
\end{equation}
With this notation,
\begin{equation}
 \Phi(W)=\beta(1-\beta)-\ip{d}{B}.
 \label{eq:B-form}
\end{equation}
For every graphon $W$, \eqref{eq:dirichlet}--\eqref{eq:B-form} remain valid
after replacing the sum in \eqref{eq:dirichlet} by
$\int_0^1\gfun(d_W(x))\dd x$ and setting
\[
 B:=W(\one-d_W)=d_W-Wd_W,
 \qquad
 B(x)=\int_0^1W(x,y)(1-d_W(y))\dd y.
\]
\end{lemma}

\begin{proof}
By \eqref{eq:weighted-parameters}, $r_i=1-d_i$ and
$\overline W_{ij}=1-W_{ij}$. Since
$\sum_i\alpha_ir_i=1-\beta$ and $\sum_j\alpha_jW_{ij}=d_i$,
\begin{align*}
 \Phi(W)
 &={}\sum_i\alpha_ir_i^2
 -\sum_{i,j}\alpha_i\alpha_j(1-W_{ij})r_ir_j\\
 &={}\sum_i\alpha_ir_i^2-(1-\beta)^2
 +\sum_{i,j}\alpha_i\alpha_jW_{ij}r_ir_j\\
 &={}\sum_i\alpha_id_i^2-\beta^2
 +\sum_{i,j}\alpha_i\alpha_jW_{ij}(1-d_i)(1-d_j)\\
 &={}\beta(1-\beta)-\sum_i\alpha_id_i^2
 +\sum_{i,j}\alpha_i\alpha_jW_{ij}d_id_j.
\end{align*}
The last line is \eqref{eq:operator-form}. Moreover,
\begin{align}
 \Ecal_W(d)
 ={}\frac12\sum_{i,j}\alpha_i\alpha_jW_{ij}(d_i-d_j)^2
 ={}\sum_i\alpha_id_i^2\sum_j\alpha_jW_{ij}
 -\sum_{i,j}\alpha_i\alpha_jW_{ij}d_id_j
 ={}\sum_i\alpha_id_i^3
 -\sum_{i,j}\alpha_i\alpha_jW_{ij}d_id_j.
\label{eq:finite-energy-algebra}
\end{align}
Substituting \eqref{eq:finite-energy-algebra} into
\eqref{eq:operator-form} gives
\eqref{eq:dirichlet}. Moreover,
$(W\one)_i=\sum_j\alpha_jW_{ij}=d_i$, so \eqref{eq:B-def} gives
$B=W(\one-d)=d-Wd$ and
$\ip{d}{B}=\|d\|_2^2-\ip{d}{Wd}$. Therefore \eqref{eq:B-form} follows
from \eqref{eq:operator-form}. For
graphons, put $r=\one-d$. Since $0\le W,d,r\le1$, every integrand below is
bounded and measurable, so the stated form of Fubini's theorem applies.
By symmetry of $W$, for almost every $x,y\in[0,1]$,
\begin{align*}
 \int_0^1\overline W(z,x)\dd z
 =\int_0^1(1-W(x,z))\dd z=1-d(x)=r(x)\quad\text{ and }\quad
 \int_0^1W(y,z)\dd z=d(y).
\end{align*}
Hence \eqref{eq:Phi-graphon} and Fubini's theorem give
\begin{align*}
 \Phi(W)
 ={}\int_{[0,1]^2}
 \left(\int_0^1\overline W(x_0,x)\dd x_0\right)
 \overline W(x,y)
 \left(\int_0^1W(y,x_3)\dd x_3\right)\dd x\dd y
 ={}\int_{[0,1]^2}\overline W(x,y)r(x)d(y)\dd x\dd y.
\end{align*}
Since $d(y)=1-r(y)$, another application of Fubini's theorem gives
\begin{align*}
 \Phi(W)
 &={}\int_{[0,1]^2}\overline W(x,y)r(x)(1-r(y))\dd x\dd y\\
 &={}\int_{[0,1]^2}\overline W(x,y)r(x)\dd x\dd y
 -\int_{[0,1]^2}\overline W(x,y)r(x)r(y)\dd x\dd y\\
 &={}\int_0^1r(x)
 \left(\int_0^1\overline W(x,y)\dd y\right)\dd x
 -\int_{[0,1]^2}\overline W(x,y)r(x)r(y)\dd x\dd y\\
 &={}\int_0^1r(x)^2\dd x
 -\int_{[0,1]^2}\overline W(x,y)r(x)r(y)\dd x\dd y.
\end{align*}
Expanding $\overline W=1-W$ and $r=1-d$, and then changing the order of
integration, yields
\begin{align*}
 \Phi(W)
 &={}\int_0^1r(x)^2\dd x-(1-\beta)^2
 +\int_{[0,1]^2}W(x,y)r(x)r(y)\dd x\dd y\\
 &={}\int_0^1d(x)^2\dd x-\beta^2
 +\int_{[0,1]^2}W(x,y)(1-d(x))(1-d(y))\dd x\dd y\\
 &={}\beta(1-\beta)-\int_0^1d(x)^2\dd x
 +\int_{[0,1]^2}W(x,y)d(x)d(y)\dd x\dd y\\
 &={}\beta(1-\beta)-\|d\|_2^2+\ip{d}{Wd}.
\end{align*}
Moreover, symmetry of $W$ and Fubini's theorem give
\begin{align}
 \Ecal_W(d)
 &={}\frac12\int_{[0,1]^2}W(x,y)(d(x)-d(y))^2\dd x\dd y\\
 &={}\int_0^1d(x)^2\left(\int_0^1W(x,y)\dd y\right)\dd x
 -\int_{[0,1]^2}W(x,y)d(x)d(y)\dd x\dd y\\
 &={}\int_0^1d(x)^3\dd x-\ip{d}{Wd}.
\label{eq:graphon-energy-algebra}
\end{align}
Substituting \eqref{eq:graphon-energy-algebra} into
\eqref{eq:operator-form} proves
\eqref{eq:dirichlet}. Finally,
$(W\one)(x)=\int_0^1W(x,y)\dd y=d(x)$ for almost every $x$, so the graphon
form of \eqref{eq:B-def} gives $B=W(\one-d)=d-Wd$ and
$\ip{d}{B}=\|d\|_2^2-\ip{d}{Wd}$. Together with
\eqref{eq:operator-form}, this proves the graphon form of
\eqref{eq:B-form}.
\end{proof}

For a finite graph $G$, put $d_v:=d_G(v)/n$ and
$\beta_G:=2e(G)/(n(n-1))$.
In the weighted graph $W_G$,
\[
 \alpha_v=\frac1n,
 \qquad
 (W_G)_{uv}=\one_{\{uv\in E(G)\}},
 \qquad
 \widehat\beta_G:=\sum_v\alpha_vd_v=\frac{2e(G)}{n^2}.
\]
Since $\beta_G=2e(G)/(n(n-1))$,
\begin{equation}
 \widehat\beta_G=\left(1-\frac1n\right)\beta_G,
 \qquad
 \widehat\beta_G(1-\widehat\beta_G)
 =\beta_G(1-\beta_G)+O\left(\frac1n\right).
\label{eq:finite-beta-normalization}
\end{equation}
Moreover,
\begin{equation}
 \sum_v\alpha_vd_v^2(1-d_v)
 =\frac1n\sum_{v\in V(G)}d_v^2(1-d_v),
\label{eq:finite-degree-term}
\end{equation}
and
\begin{align}
 \Ecal_{W_G}(d)
 ={}\frac12\sum_{u,v}\alpha_u\alpha_v
 (W_G)_{uv}(d_u-d_v)^2
 ={}\frac1{2n^2}\sum_{u,v}
 \one_{\{uv\in E(G)\}}(d_u-d_v)^2
 ={}\frac1{n^2}\sum_{uv\in E(G)}(d_u-d_v)^2,
\label{eq:finite-energy-term}
\end{align}
since every unordered edge occurs twice in the sum over ordered pairs.
Substituting \eqref{eq:finite-beta-normalization},
\eqref{eq:finite-degree-term}, and \eqref{eq:finite-energy-term} into
\eqref{eq:dirichlet}, and then using
\eqref{eq:finite-to-weighted}, gives
\begin{align}
 p(H_3,G)
 ={}&\beta_G(1-\beta_G)-\frac1n\sum_{v\in V(G)}d_v^2(1-d_v)-\frac1{n^2}\sum_{uv\in E(G)}(d_u-d_v)^2+O\left(\frac1n\right).
 \label{eq:finite-dirichlet}
\end{align}
The last term penalizes edges whose endpoints have different normalized degrees.

\section{Proof of the split-regular bounds}\label{sec:split-regular}

We use the following form of Jensen's inequality; see \cite{Folland}.
\begin{theorem}\label{thm:jensen}
Let $I\subseteq\mathbb R$ be an interval, let $D\subseteq[0,1]$ be measurable
with $|D|>0$, and let $f:D\to I$ be integrable. If $q:I\to\mathbb R$ is
convex and $q\circ f$ is integrable, then
\[
 \int_Dq(f(x))\dd x
 \ge |D|q\left(\frac1{|D|}\int_Df(x)\dd x\right).
\]
For a finite weighted set $(D,\alpha)$, the corresponding form is
\[
 \sum_{i\in D}\alpha_iq(f_i)
 \ge |D|q\left(\frac1{|D|}\sum_{i\in D}\alpha_if_i\right),
 \qquad |D|=\sum_{i\in D}\alpha_i>0.
\]
\end{theorem}

Let $W$ be the disjoint union of a clique $A$ of weight $a$ and a regular
weighted graph on $B$ of weight $1-a$ and absolute degree $x\le1-a$. Then
$d=a$ on $A$, $d=x$ on $B$, and $\Ecal_W(d)=0$. Hence
$\Phi(W)=F(a,x)$, where $F$ is defined in \eqref{eq:F-construction}. We first
record the convex envelope used below. Recall that the lower convex envelope
of a function $f$ on an interval is the pointwise largest convex function
that is everywhere at most $f$.

\begin{lemma}\label{lem:psi-envelope}
The lower convex envelope of $\gfun(t)=t^2(1-t)$ on $[0,1/2]$ is
\begin{equation}
 \gcvx(t)=
 \begin{cases}
  t^2(1-t),&0\le t\le1/4,\\[1mm]
  \dfrac{5t}{16}-\dfrac1{32},&1/4\le t\le1/2.
 \end{cases}
 \label{eq:g-envelope}
\end{equation}
\end{lemma}

\begin{proof}
Since $\gfun''(t)=2-6t>0$ on $[0,1/4]$, the first piece is convex. The
line through $(1/4,\gfun(1/4))$ and $(1/2,\gfun(1/2))$ is
$\ell(t)=5t/16-1/32$, and
\begin{equation}
 \gfun(t)-\ell(t)
 =\frac{(1-2t)(4t-1)^2}{32}.
 \label{eq:g-envelope-cert}
\end{equation}
The right-hand side is nonnegative on $[1/4,1/2]$, and
$\ell'=5/16=\gfun'(1/4)$; hence $\gcvx$ is a convex minorant of $\gfun$.
If $q\le\gfun$ is convex, then $q\le\gfun=\gcvx$ on $[0,1/4]$, while on
$[1/4,1/2]$ convexity places $q$ below the chord through its endpoint values,
which is at most $\ell=\gcvx$. Thus every convex minorant is at most $\gcvx$.
\end{proof}

Consequently, if $D$ has positive weight and $f:D\to[0,1/2]$, then
\begin{equation}
 \int_D\gfun(f)\ge\int_D\gcvx(f)
 \ge |D|\gcvx\left(\frac1{|D|}\int_D f(x)\dd x\right),
 \label{eq:envelope-jensen}
\end{equation}
where the second inequality follows from Theorem~\ref{thm:jensen}.

\begin{lemma}\label{lem:regular-max}
For $1/2\le a\le1$ and $0\le x\le1-a$, we have
\[
 \Freg(a,x)\le\pstar.
\]
The unique interior global maximizer with $x<1/4$ is $(a_*,x_*)$.
\end{lemma}

\begin{proof}
Put $\beta=a^2+(1-a)x$. Rewriting \eqref{eq:F-construction} using
\eqref{eq:g-def} gives
\begin{equation}
 F(a,x)=\beta(1-\beta)-a\gfun(a)-(1-a)\gfun(x).
 \label{eq:F-dirichlet}
\end{equation}
Since $\gcvx\le\gfun$ by Lemma~\ref{lem:psi-envelope} and $1-a\ge0$, define
\begin{equation}
 \widetilde F(a,x)
 :=\beta(1-\beta)-a\gfun(a)-(1-a)\gcvx(x),
 \qquad F(a,x)\le\widetilde F(a,x).
 \label{eq:F-tilde}
\end{equation}
If $x\ge1/4$, then \eqref{eq:g-envelope} gives
$\gcvx(x)=5x/16-1/32$. Hence direct differentiation of
\eqref{eq:F-tilde} gives
\begin{equation}
 \frac{\partial\widetilde F}{\partial x}
 =\frac{1-a}{16}\bigl(11-32a^2-32(1-a)x\bigr)\le0,
 \label{eq:F-tilde-derivative}
\end{equation}
since for $x\ge1/4$ the bracket is at most
$3+8a-32a^2\le-1$. Thus \eqref{eq:F-tilde-derivative} gives
$\partial\widetilde F/\partial x<0$ whenever $a<1$. Moreover,
\[
 \frac{19}{128}-\widetilde F\left(a,\frac14\right)
 =\frac{(2a-1)(32a^2-12a-1)}{128}\ge0.
\]
Consequently, for $x\ge1/4$,
\[
 F(a,x)\le\widetilde F(a,x)
 \le\widetilde F\left(a,\frac14\right)
 \le\frac{19}{128}<\frac3{20}.
\]
Thus an optimum larger than $3/20$ must have $x<1/4$.

Differentiating \eqref{eq:F-dirichlet} gives
\begin{align*}
 \frac{\partial\Freg}{\partial x}
 &=(1-a)\bigl(1-2\beta-\gfun'(x)\bigr)=(a-1)(2a^2-2ax+4x-3x^2-1),\\
 \frac{\partial\Freg}{\partial a}
 ={}&(2a-x)(1-2\beta)-\gfun(a)-a\gfun'(a)+\gfun(x)\\
 ={}&(-2a+x)(2a^2-2ax+4x-3x^2-1)+(a-x)(4a^2+4ax-3a-2x^2+x).
\end{align*}
In the interior, $a<1$ and $0<x<1-a<a$. Therefore
the derivative equations give the following system, whose two left-hand sides
we denote by $C_1$ and $C_2$, respectively:
\begin{equation}
 2a^2-2ax+4x-3x^2-1=0,
 \qquad 4a^2+4ax-3a-2x^2+x=0.
 \label{eq:critical-system}
\end{equation}
Subtracting twice the first equation in \eqref{eq:critical-system} from the
second gives
\begin{equation}
 (8x-3)a+4x^2-7x+2=0.
\label{eq:critical-linear}
\end{equation}
Since $x<1/4$, the denominator $3-8x$ is nonzero, and therefore
\begin{equation}
 a=\frac{4x^2-7x+2}{3-8x}.
 \label{eq:a-of-x}
\end{equation}
Substituting $a=(4x^2-7x+2)/(3-8x)$ from \eqref{eq:a-of-x} into the first
equation in \eqref{eq:critical-system} and multiplying by $(3-8x)^2$ gives
\begin{align}\label{eq:resultant}
 0={}&(3-8x)^2(2a^2-2ax+4x-3x^2-1)=-(3x-1)(32x^3-40x^2+13x-1).
\end{align}
Write $C(x)=32x^3-40x^2+13x-1$. Its two critical points are
$(10\pm\sqrt{22})/24$. The first lies in $(1/5,1/4)$ and the second is larger
than $1/4$. Since
$
 C(0)=-1, C(1/5)=\frac{32}{125}>0$
 and $C(1/4)=\frac14>0,
$
the polynomial increases up to its first critical point and then decreases
but remains positive through $x=1/4$. Hence it has exactly one root in
$(0,1/4)$, namely $x_*$. Since $3x_*-1\ne0$,
\eqref{eq:resultant} gives $C_1=0$ at $x=x_*$ and $a=a_*$. Equation
\eqref{eq:critical-linear}, equivalently $C_2-2C_1=0$, then gives $C_2=0$,
so $(a_*,x_*)$ is a critical point. It is
feasible and interior since, for $0<x<1/4$,
\[
 a-\frac12=\frac{(1-4x)(1-2x)}{2(3-8x)}>0,
 \qquad
 1-a-x=\frac{(1-2x)^2}{3-8x}>0.
\]

It remains to inspect the boundary of the region $x<1/4$. At $x=0$,
\[
 \Freg(a,0)=a^2(1-a)\le\frac4{27}.
\]
At $a=1/2$ we have
\begin{equation}
 16\frac{\partial\Freg}{\partial x}
 =24x^2-24x+4,
\label{eq:a-half-derivative}
\end{equation}
so \eqref{eq:a-half-derivative} shows that the maximum on
$0\le x\le1/4$ occurs at
$x=(3-\sqrt3)/6$ and equals $1/8+\sqrt3/72$. At $x=1/4$,
\eqref{eq:g-envelope-cert} is an equality, and
\[
 \frac{19}{128}-\Freg\left(a,\frac14\right)
 =\frac{(2a-1)(32a^2-12a-1)}{128}\ge0.
\]
Finally, on $x=1-a\le1/4$,
\[
 \Freg(a,1-a)=a(1-a)(2a^2-2a+1),
 \qquad
 \frac{\mathrm d}{\mathrm d a}\Freg(a,1-a)=-(2a-1)^3,
\]
so the maximum for $3/4\le a\le1$ is $15/128$. At $a=1$ the value is $0$.
All four boundary bounds are strictly below $3/20$. On the other hand,
\begin{equation}
 \Freg\left(\frac35,\frac9{80}\right)
 =\frac{38421}{256000}
 =\frac3{20}+\frac{21}{256000}>\frac3{20}.
 \label{eq:rational-benchmark}
\end{equation}
Therefore the unique feasible interior point $(a_*,x_*)$ is the global maximizer.
\end{proof}

We also need the following extension to $x\in[0,1/2]$. Define
\begin{equation}
 F_0(M,x):=\beta_0(1-\beta_0)-M\gfun(M)-(1-M)\gcvx(x),
 \qquad \beta_0=M^2+(1-M)x.
 \label{eq:F0}
\end{equation}

\begin{lemma}\label{lem:F0-extended}
For $\frac12\le M\le1$ and $0\le x\le\frac12$, we have
\begin{equation}
 F_0(M,x)\le\pstar.
 \label{eq:F0-bound}
\end{equation}
More precisely, $F_0(M,x)\le19/128$ when $x\ge1/4$, and
$F_0(M,x)\le15/128$ when $x\le1/4$ but $x>1-M$.
\end{lemma}

\begin{proof}
If $x\le1/4$ and $x\le1-M$, then \eqref{eq:F0} gives
$F_0(M,x)=\Freg(M,x)$ and
Lemma~\ref{lem:regular-max} applies. If $x\le1/4$ and $x>1-M$,
then $M>3/4$ and
\begin{equation}
 \frac{\partial F_0}{\partial x}
 =(1-M)(-2M^2+2Mx+3x^2-4x+1)\le0.
\label{eq:F0-low-derivative}
\end{equation}
The bracket is convex in $x$ and at $x=1-M,1/4$ equals, respectively,
\begin{equation}
 -M^2,
 \qquad -\frac{32M^2-8M-3}{16}.
\label{eq:F0-endpoint-signs}
\end{equation}
Both quantities in \eqref{eq:F0-endpoint-signs} are negative for $M>3/4$.
Thus equality in \eqref{eq:F0-low-derivative} can occur only in the
degenerate case $M=1$, and
\[
 F_0(M,x)\le F_0(M,1-M)
 =M(1-M)(2M^2-2M+1)\le\frac{15}{128}.
\]
Here the last one-variable polynomial has derivative $-(2M-1)^3$ and the
relevant range is $M>3/4$. Thus its maximum on the closure is its value at
$M=3/4$.
Finally, if $x\ge1/4$, then
\begin{equation}
 \frac{\partial F_0}{\partial x}
 =\frac{1-M}{16}\bigl(11-32M^2-32(1-M)x\bigr)\le0,
\label{eq:F0-high-derivative}
\end{equation}
since the bracket is at most $3+8M-32M^2\le-1$. Again equality can occur only
at $M=1$. Thus \eqref{eq:F0-high-derivative} shows that the maximum is
attained at
$x=1/4$ (or is constant in the degenerate case), and
\[
 \frac{19}{128}-F_0\left(M,\frac14\right)
 =\frac{(2M-1)(32M^2-12M-1)}{128}\ge0.
\]
Indeed, both factors are nonnegative for $M\ge1/2$; the second is increasing
there and equals $1$ at $M=1/2$.
\end{proof}

\section{Proof of the structural reduction}\label{sec:structural-reduction}

\subsection{Degree quotients}\label{subsec:chosen-extremal}

By Lemma~\ref{lem:sequence-bridge}, it suffices to maximize $\Phi$ over
graphons. More precisely, we work in the compact graphon space modulo
measure-preserving weak isomorphism, endowed with the cut metric. The
functional $\Phi$ is a finite linear combination of homomorphism densities
and is therefore cut-continuous. Hence its set of maximizers is a closed,
and thus compact, subset of this quotient space. The degree-square functional
\begin{equation}
 \Dcal(W):=\|d_W\|_2^2
 =\begin{cases}
   \sum_i\alpha_id_i^2,&W\text{ finite weighted},\\
   \int_0^1d_W(x)^2\dd x,&W\text{ a graphon}.
  \end{cases}
 \label{eq:normsq}
\end{equation}
is the homomorphism density of a two-edge path and is therefore
cut-continuous and invariant under weak isomorphism. It consequently attains
its minimum on the compact set of maximizers of $\Phi$. Choose a maximizer
$W$ attaining this minimum.
We call $W$ a \emph{chosen extremal graphon}.

\begin{theorem}\label{thm:structural-reduction}
There exists a chosen extremal graphon admitting a measurable partition
$[0,1]=A\sqcup B$, with $a:=|A|>1/2$, such that, modulo null sets,
\begin{equation}
 W=K_A\sqcup W_B,
 \qquad d_W=a\ \text{on }A,
 \qquad d_W\le\frac12\ \text{on }B.
 \label{eq:structural-reduction}
\end{equation}
Equivalently, $W=1$ on $A^2$, $W=0$ on $A\times B$, and every point of
$B$ has degree at most $1/2$.
\end{theorem}

Let $\mathcal A=\sigma(d)$ and replace $W$ by the degree quotient
\begin{equation}
 \widehat W:=\mathbb E[W\mid\mathcal A\otimes\mathcal A].
 \label{eq:degree-quotient}
\end{equation}
For a finite weighted graph, this operation averages each block between two
degree classes. Integrating $\widehat W$ in either variable gives
$\mathbb E[d\mid\mathcal A]=d$, and testing the conditional expectation
against the $\mathcal A\otimes\mathcal A$-measurable function
$(x,y)\mapsto d(x)d(y)$ shows
$\ip{d}{\widehat Wd}=\ip{d}{Wd}$. Thus \eqref{eq:operator-form} implies that
the quotient \eqref{eq:degree-quotient} preserves $d$, $\Phi$, and
$\Dcal(W)$. In particular,
$\widehat W$ is again a chosen extremal graphon, and we replace $W$ by it.
Henceforth, almost everywhere,
\begin{equation}
 W(x,y)=w(d(x),d(y)),
 \qquad B(x)=B(d(x)).
 \label{eq:degree-measurable}
\end{equation}
Put $\mathbf d(x,y):=(d(x),d(y))$. Before completion,
$\mathcal A\otimes\mathcal A=\sigma(\mathbf d)$. A bounded function
measurable with respect to the completion of this $\sigma$-algebra has an
uncompleted measurable version, and the Doob--Dynkin factorization gives a
Borel function of $\mathbf d$. Applying the same argument to $\sigma(d)$,
we obtain Borel functions $w:[0,1]^2\to[0,1]$ and
$b:[0,1]\to\mathbb R$ such that
$W(x,y)=w(d(x),d(y))$ and $B(x)=b(d(x))$ almost everywhere.
Symmetrizing $w$ preserves this identity. Below we write $B(t)$ for $b(t)$.
Let $\mathcal P$ be a partition of $V$. If $i$ lies in a block
$C\in\mathcal P$, define
\[
 (Pf)_i:=\frac1{|C|}\sum_{u\in C}\alpha_uf_u.
\]
For $i\in C$ and $j\in D$, replace $W_{ij}$ by the weighted edge density
between $C$ and $D$,
\[
 W^P_{ij}:=\frac1{|C||D|}
 \sum_{u\in C,\,v\in D}\alpha_u\alpha_vW_{uv}.
\]
Set $z=d-Pd$. Then $d_{W^P}=Pd$ and $\sum_i\alpha_i(Pd)_i=\beta$.
For graphons, let $\mathcal G$ be any sub-$\sigma$-algebra of the Lebesgue
$\sigma$-algebra, and define
\begin{equation}
 Pf:=\mathbb E[f\mid\mathcal G],
 \qquad
 W^P:=\mathbb E[W\mid\mathcal G\otimes\mathcal G].
 \label{eq:general-quotient}
\end{equation}
The transpose map preserves $\mathcal G\otimes\mathcal G$, so the conditional
expectation may be chosen symmetric. Its marginal is
$\mathcal G$-measurable. For every bounded $\mathcal G$-measurable test
function $g$, conditional expectation and Fubini's theorem give
\begin{equation}
 \int_0^1g(x)\left(\int_0^1W^P(x,y)\dd y\right)\dd x
 =\int_{[0,1]^2}g(x)W(x,y)\dd x\dd y
 =\int_0^1g(x)d(x)\dd x.
\label{eq:quotient-marginal-test}
\end{equation}
Consequently, \eqref{eq:quotient-marginal-test} gives
$d_{W^P}=Pd$ almost everywhere. Thus this definition includes
both finite measurable block quotients and the quotients used below that
collapse one measurable set while retaining the full measurable structure on
its complement.

\begin{lemma}\label{lem:quotient-identity}
For every finite weighted graph quotient and, in the graphon setting, for
every quotient \eqref{eq:general-quotient} associated with a sub-$\sigma$-algebra
$\mathcal G$, we have
\begin{equation}
 \Phi(W^P)-\Phi(W)
 =2\ip{z}{B}-\|z\|_2^2+\ip{z}{Wz},
 \label{eq:quotient-identity}
\end{equation}
and
\begin{equation}
 \Dcal(W^P)=\Dcal(W)-\|z\|_2^2.
 \label{eq:normsq-drop}
\end{equation}
\end{lemma}

\begin{proof}
The quotient preserves the total edge density $\beta$. Since $Pd$ is
$\mathcal G$-measurable in the graphon setting, conditional expectation
against the $\mathcal G\otimes\mathcal G$-measurable function
$(x,y)\mapsto Pd(x)Pd(y)$ gives
\begin{equation}
 \ip{Pd}{W^PPd}=\ip{Pd}{WPd}.
\label{eq:quotient-bilinear}
\end{equation}
The same identity follows by direct weighted summation in the finite model.
Using $Pd=d-z$, \eqref{eq:quotient-bilinear}, and
\eqref{eq:operator-form}, we obtain
\begin{align*}
 \Phi(W^P)-\Phi(W)
 &=-\|d-z\|_2^2+\|d\|_2^2
   +\ip{d-z}{W(d-z)}-\ip{d}{Wd}\\
 &=2\ip{d}{z}-\|z\|_2^2
   -2\ip{z}{Wd}+\ip{z}{Wz}\\
 &=2\ip{z}{d-Wd}-\|z\|_2^2+\ip{z}{Wz}.
\end{align*}
Since $B=d-Wd$ by \eqref{eq:B-def}, this is
\eqref{eq:quotient-identity}. Finally, $P$ is an orthogonal projection and
$z=d-Pd$ is orthogonal to $Pd$. Hence
\begin{equation}
 \|d\|_2^2=\|Pd\|_2^2+\|z\|_2^2.
\label{eq:quotient-pythagoras}
\end{equation}
Equation \eqref{eq:quotient-pythagoras} is exactly
\eqref{eq:normsq-drop} since $d_{W^P}=Pd$.
\end{proof}

\begin{lemma}\label{lem:quotient-rigid}
For the chosen extremal graphon,
\begin{equation}
 Pd\ne d\quad\Longrightarrow\quad\Phi(W^P)<\Phi(W).
 \label{eq:strict-quotient}
\end{equation}
\end{lemma}

\begin{proof}
If equality held, then $W^P$ would also maximize $\Phi$, while
\eqref{eq:normsq} and \eqref{eq:normsq-drop} would give a smaller value of
$\Dcal(W)$.
\end{proof}

\subsection{The half-slope inequality}\label{sec:half-slope}

Let $A,C$ be disjoint vertex sets of weight $\delta$, and let $P$ merge
$A\cup C$ into one quotient block while fixing all other vertices. Put
\begin{align*}
 r&=\frac1\delta\sum_{i\in A}\alpha_i d_i,
 &u&=\frac1\delta\sum_{i\in C}\alpha_i d_i,
 &\Delta&=u-r,\\
 b_A&=\frac1\delta\sum_{i\in A}\alpha_i B_i,
 &b_C&=\frac1\delta\sum_{i\in C}\alpha_i B_i.
\end{align*}
and
\[
 V_A=\sum_{i\in A}\alpha_i(d_i-r)^2,
 \quad V_C=\sum_{i\in C}\alpha_i(d_i-u)^2,
\]
\[
 \Gamma_A=\sum_{i\in A}\alpha_i(d_i-r)(B_i-b_A),
 \quad \Gamma_C=\sum_{i\in C}\alpha_i(d_i-u)(B_i-b_C).
\]
For graphons, replace the sums by integrals over measurable sets. Let
$\mu=d_\#\lambda$ be the degree distribution. By
\eqref{eq:degree-measurable}, $B$ is a function of the degree variable.

\begin{lemma}\label{lem:two-set-quotient}
With $z=d-Pd$,
\begin{align}
 \Phi(W^P)-\Phi(W)
 ={}&\delta\Delta\left(b_C-b_A-\frac\Delta2\right)
 +2(\Gamma_A+\Gamma_C)-(V_A+V_C)+\ip{z}{Wz}.
 \label{eq:two-set-quotient}
\end{align}
Moreover,
\begin{equation}
 |\ip{z}{Wz}|\le\|z\|_1^2\le4\delta^2.
 \label{eq:quadratic-error}
\end{equation}
\end{lemma}

\begin{proof}
The common quotient value of $d$ on $A\cup C$ is $(r+u)/2$. Hence, on
$A$ and $C$, respectively,
\begin{equation}
 z=(d-r)-\frac\Delta2,
 \qquad z=(d-u)+\frac\Delta2.
\label{eq:two-set-z}
\end{equation}
Since the centered terms have integral zero on their corresponding sets,
\begin{align*}
 \ip{z}{B}
 =\Gamma_A-\frac{\delta\Delta}{2}b_A
   +\Gamma_C+\frac{\delta\Delta}{2}b_C
   \quad\text{ and }\quad
 \|z\|_2^2
 =V_A+V_C+2\delta\frac{\Delta^2}{4}
  =V_A+V_C+\frac{\delta\Delta^2}{2}.
\end{align*}
Substituting \eqref{eq:two-set-z} into \eqref{eq:quotient-identity} gives
\eqref{eq:two-set-quotient}. Moreover, $z$ is supported on $A\cup C$ and
$|z|\le1$. Thus, using $0\le W\le1$,
\[
 |\ip{z}{Wz}|
 \le\int\!\!\int |z(x)||z(y)|\dd x\dd y
 =\|z\|_1^2\le(2\delta)^2,
\]
which proves \eqref{eq:quadratic-error}.
\end{proof}

\begin{lemma}\label{lem:half-slope}
For $\mu^2$-almost every $r<u$,
\begin{equation}
 B(u)-B(r)\le\frac{u-r}{2}.
 \label{eq:half-slope}
\end{equation}
Equivalently,
\begin{equation}
 \chi(t):=\frac t2-B(t)
 \label{eq:chi}
\end{equation}
has a nondecreasing representative.
\end{lemma}

\begin{proof}
Suppose first that the degree classes
$\{d=r\}$ and $\{d=u\}$ have positive weight, where $r<u$. Splitting a
weighted vertex into identical clones does not change any degree, edge
density, or value of $\Phi$. We may therefore choose equal-weight subsets
$A\subseteq\{d=r\}$ and $C\subseteq\{d=u\}$ of arbitrarily small weight
$\delta$. On these sets $d$ and $B$ are constant, so
$V_A=V_C=\Gamma_A=\Gamma_C=0$. The two-set quotient formula becomes
\begin{equation}
 \Phi(W^P)-\Phi(W)
 =\delta(u-r)\left(B(u)-B(r)-\frac{u-r}{2}\right)+O(\delta^2).
\label{eq:atomic-half-slope-variation}
\end{equation}
If the bracket were positive, a sufficiently small $\delta$ would produce a
larger admissible quotient by \eqref{eq:atomic-half-slope-variation},
contradicting maximality. This proves
\eqref{eq:half-slope} whenever both degree values are positive-weight degree
classes. The same argument applies to two positive-measure graphon degree
fibers.

It remains to treat pairs from the non-atomic part of the degree
distribution. Since $B\in L^\infty(\mu)$ and $\mu$ is a finite Borel measure
on $[0,1]$, the one-dimensional Lebesgue--Besicovitch differentiation theorem
for finite Borel measures implies that $B$ has a $\mu$-Lebesgue value at
$\mu$-almost every degree; see, for example, \cite[Chapter~3]{Folland}.
Choose two such non-atomic points $r_0<u_0$.

Let $\eta_k\downarrow0$. At the two Lebesgue points, we may choose radii
$0<\rho_k^r,\rho_k^u\le\eta_k$
such that the intervals
\[
 I_k^r=(r_0-\rho_k^r,r_0+\rho_k^r)\cap[0,1],
 \qquad
 I_k^u=(u_0-\rho_k^u,u_0+\rho_k^u)\cap[0,1]
\]
are disjoint, have positive $\mu$-measure, and satisfy
\begin{align}
 \frac{\mu\{t\in I_k^r:|B(t)-B(r_0)|>\eta_k\}}
      {\mu(I_k^r)}
 <\eta_k
 \quad\text{ and }\quad
 \frac{\mu\{t\in I_k^u:|B(t)-B(u_0)|>\eta_k\}}
      {\mu(I_k^u)}
 <\eta_k.
 \label{eq:Lebesgue-slices}
\end{align}
Choose the radii after fixing $\eta_k$, so that the corresponding normalized
integrals of $|B-B(r_0)|$ and $|B-B(u_0)|$ are below $\eta_k^2$;
\eqref{eq:Lebesgue-slices} then follows from Markov's inequality.

Let $G_k^r$ and $G_k^u$ be the corresponding good portions, on which the
respective inequalities $|B(t)-B(r_0)|\le\eta_k$ and
$|B(t)-B(u_0)|\le\eta_k$ hold. By \eqref{eq:Lebesgue-slices}, both good
portions have positive $\mu$-measure. Since the underlying vertex space
$([0,1],\lambda)$ is non-atomic, we can choose
$A_k\subseteq d^{-1}(G_k^r)$ and
$C_k\subseteq d^{-1}(G_k^u)$ with common measure
$0<|A_k|=|C_k|=:\delta_k\le\eta_k$.
Write $r_k,u_k,b_{A_k},b_{C_k}$ for the averages used in
Lemma~\ref{lem:two-set-quotient}, and put $\Delta_k=u_k-r_k$. By construction,
\begin{align*}
 r_k&=r_0+O(\eta_k),&
 u_k&=u_0+O(\eta_k),\\
 b_{A_k}&=B(r_0)+O(\eta_k),&
 b_{C_k}&=B(u_0)+O(\eta_k).
\end{align*}
Moreover, on each selected set both the degree and the value of $B$ vary by
$O(\eta_k)$. Therefore
\[
 V_{A_k}+V_{C_k}
 +|\Gamma_{A_k}|+|\Gamma_{C_k}|
 =O(\eta_k^2\delta_k).
\]
For the remaining quadratic term we use the uniform estimate
\eqref{eq:quadratic-error} directly:
\[
 |\ip{z}{Wz}|\le4\delta_k^2.
\]
Maximality and \eqref{eq:two-set-quotient} now give
\begin{equation}
0\ge
 \delta_k\Delta_k
 \left(b_{C_k}-b_{A_k}-\frac{\Delta_k}{2}\right)
 +O(\eta_k^2\delta_k+\delta_k^2).
\label{eq:approx-half-slope}
\end{equation}
After division by $\delta_k$, the error is
$O(\eta_k^2+\delta_k)=o(1)$ since $\delta_k\le\eta_k$. Letting
$k\to\infty$ in \eqref{eq:approx-half-slope} yields
\begin{equation}
 (u_0-r_0)
 \left(B(u_0)-B(r_0)-\frac{u_0-r_0}{2}\right)\le0.
\label{eq:half-slope-limit}
\end{equation}
Since $u_0>r_0$, \eqref{eq:half-slope-limit} implies
\eqref{eq:half-slope} for such a pair.

If exactly one of the two degree values is an atom of $\mu$, choose the set
on the atomic side inside its exact degree fiber and choose the other set
inside the corresponding good portion above, with the same measure
$\delta_k\le\eta_k$. The variance and covariance on the exact degree fiber
vanish, and the preceding estimates on the non-atomic side are unchanged.
The case of two atoms was proved at the start. Hence
\eqref{eq:half-slope} holds for $\mu^2$-almost every ordered pair $r<u$.

For the representative assertion, suppose that a measurable
function $f$ satisfies $f(r)\le f(u)$ for $\mu^2$-almost every $r<u$. For
$q\in\mathbb Q$, put $E_q=\{f>q\}$. The assumed order relation gives
\begin{equation}
 \mu^2\{(r,u):r<u,\ r\in E_q,\ u\notin E_q\}=0.
\label{eq:no-inversion}
\end{equation}
If $0<\mu(E_q)<1$, let
$s_q:=\inf\{s:\mu(E_q\cap[0,s])>0\}$. Then
$E_q\cap[0,s_q)$ is null, while \eqref{eq:no-inversion} forces
$(s,1]\setminus E_q$ to be null for every $s>s_q$. Thus $E_q$ agrees
$\mu$-almost everywhere with an upper ray. At an atom $\{s_q\}$, include
the threshold exactly when $E_q$ has full measure in that atom. Since
$E_{q_2}\subseteq E_{q_1}$ for $q_1<q_2$, these upper-ray representatives
$E_q^\ast$ may be chosen nested; the cases $\mu(E_q)\in\{0,1\}$ are
included.
The function
$\widetilde f(t):=\sup\{q\in\mathbb Q:t\in E_q^\ast\}$,
where $E_q^\ast$ are these nested versions, is nondecreasing and equals $f$
$\mu$-almost everywhere. Apply this construction to $f=\chi$. We fix the
resulting representative and set $B(t)=t/2-\chi(t)$ at any remaining
$\mu$-null degree values.
\end{proof}

\begin{lemma}\label{lem:ordered-constancy}
Let $\nu$ be a finite Borel measure on $[0,1]$, and let $f$ be
nondecreasing. If
\[
 f(r)=f(u)
 \qquad\text{for $\nu^2$-almost every pair }r<u,
\]
then $f$ is constant $\nu$-almost everywhere.
\end{lemma}

\begin{proof}
If $f$ were not essentially constant, there would be a real number $q$ such
that both $\{f\le q\}$ and $\{f>q\}$ have positive $\nu$-measure. For every
$r$ in the first set and $u$ in the second set, monotonicity implies
$r<u$, while $f(r)\ne f(u)$. Their Cartesian product therefore has positive
$\nu^2$-measure, contradicting the hypothesis.
\end{proof}

\subsection{Edge variations and high-degree rows}\label{sec:fiber-splitting}

For a finite weighted red graph, let $K=(K_{ij})$ be a symmetric perturbation
of the edge-density matrix and write
\[
 k_i:=\sum_j\alpha_jK_{ij},
 \qquad \eta:=\sum_i\alpha_ik_i.
\]

\begin{lemma}\label{lem:variation}
For every sufficiently small $\varepsilon$ such that
$0\le W_{ij}+\varepsilon K_{ij}\le1$ for all $i,j$,
\begin{align}
 \Phi(W+\varepsilon K)-\Phi(W)
 ={}&\varepsilon\sum_{i,j}\alpha_i\alpha_jL_{ij}K_{ij}
 +\varepsilon^2\mathcal Q_W(K)\notag\\
 &\quad+\varepsilon^3\ip{k}{Kk},
 \label{eq:variation}
\end{align}
where
\begin{equation}
 L_{ij}=1-2\beta+d_id_j-B_i-B_j,
 \label{eq:Euler-kernel}
\end{equation}
and
\begin{equation}
 \mathcal Q_W(K)
 =-\eta^2-\|k\|_2^2+\ip{k}{Wk}+2\ip{k}{Kd}.
 \label{eq:Qform}
\end{equation}
The same expansion holds for bounded graphon perturbations after replacing
the first sum by the corresponding double integral.
\end{lemma}

\begin{proof}
Insert $d_\varepsilon=d+\varepsilon k$ and
$\beta_\varepsilon=\beta+\varepsilon\eta$ into \eqref{eq:operator-form}.
The three terms in \eqref{eq:operator-form} expand as
\begin{align*}
 \beta_\varepsilon(1-\beta_\varepsilon)
 &=\beta(1-\beta)+\varepsilon(1-2\beta)\eta
   -\varepsilon^2\eta^2,\\
 -\|d_\varepsilon\|_2^2
 &=-\|d\|_2^2-2\varepsilon\ip{d}{k}
   -\varepsilon^2\|k\|_2^2,\\
 \ip{d_\varepsilon}{(W+\varepsilon K)d_\varepsilon}
 &=\ip{d}{Wd}
   +\varepsilon\bigl(2\ip{k}{Wd}+\ip{d}{Kd}\bigr)\\
 &\quad+\varepsilon^2\bigl(\ip{k}{Wk}+2\ip{k}{Kd}\bigr)
   +\varepsilon^3\ip{k}{Kk}.
\end{align*}
Here symmetry of $W$ and $K$ was used. The linear coefficient is
\[
 (1-2\beta)\eta-2\ip{k}{d-Wd}+\ip{d}{Kd}.
\]
Since $k_i=\sum_j\alpha_jK_{ij}$ and $K_{ij}=K_{ji}$, this coefficient equals
\begin{align*}
 \sum_{i,j}\alpha_i\alpha_j
 \bigl(1-2\beta-B_i-B_j+d_id_j\bigr)K_{ij}
 =\sum_{i,j}\alpha_i\alpha_jL_{ij}K_{ij}.
\end{align*}
Collecting the quadratic and cubic coefficients proves
\eqref{eq:variation}--\eqref{eq:Qform}. There are no higher-order terms
since \eqref{eq:operator-form} is cubic in $W$.
\end{proof}

\begin{lemma}\label{lem:first-order}
For a finite weighted maximizer, the following conditions hold entrywise:
\begin{equation}
 W_{ij}=0\Rightarrow L_{ij}\le0,
 \qquad W_{ij}=1\Rightarrow L_{ij}\ge0,
 \qquad0<W_{ij}<1\Rightarrow L_{ij}=0.
 \label{eq:KKT}
\end{equation}
For a graphon maximizer, define
$L(x,y):=1-2\beta+d(x)d(y)-B(x)-B(y)$.
This is the graphon version of \eqref{eq:Euler-kernel}. The analogous
implications in \eqref{eq:KKT} hold for almost every
$(x,y)\in[0,1]^2$.
\end{lemma}

\begin{proof}
In the finite setting, at an entry with $W_{ij}=0$ only positive first-order
changes are feasible, so maximality gives $L_{ij}\le0$. At an entry with
$W_{ij}=1$ only negative changes are feasible, giving $L_{ij}\ge0$. At a
fractional entry both signs are feasible, giving $L_{ij}=0$.

We give the measure-theoretic argument for graphons. Both $W$ and $L$ are
symmetric. Suppose first that, for some $\gamma>0$, the symmetric set
$E=\{(x,y):W(x,y)=0,\ L(x,y)\ge\gamma\}$
has positive product measure. Take $K=\one_E$. For every sufficiently small
$\varepsilon>0$, the graphon $W+\varepsilon K$ is feasible. Its linear term
in \eqref{eq:variation} is at least $\gamma|E|\varepsilon$, whereas the
quadratic and cubic terms are $O(\varepsilon^2)$ since $K$ is bounded.
Thus $\Phi(W+\varepsilon K)>\Phi(W)$ for small $\varepsilon$, a
contradiction. Taking the countable union over rational $\gamma>0$ proves
$L\le0$ almost everywhere on $\{W=0\}$. Applying the same argument to
$E=\{(x,y):W(x,y)=1,\ L(x,y)\le-\gamma\}$ with $K=-\one_E$
proves $L\ge0$ almost everywhere on $\{W=1\}$.

Finally, suppose that $L$ is nonzero on a positive-measure subset of
$\{0<W<1\}$. By a countable exhaustion, there exist integers $j,\ell\ge1$
and a sign $\sigma\in\{-1,1\}$ such that the symmetric set
\[
 E=\left\{(x,y):
 \frac1j\le W(x,y)\le1-\frac1j,\quad
 \sigma L(x,y)\ge\frac1\ell\right\}
\]
has positive measure. Set $K=\sigma\one_E$. For
$0<\varepsilon<1/j$, the perturbation $W+\varepsilon K$ is feasible, its
linear contribution is at least $\varepsilon|E|/\ell$, and its remaining
terms are $O(\varepsilon^2)$. This again contradicts maximality. Hence
$L=0$ almost everywhere on the fractional region.
\end{proof}

For a fixed first-end degree $z$, using $B(u)=u/2-\chi(u)$ from
\eqref{eq:chi} gives
\begin{equation}
 L_z(u)=1-2\beta-B(z)+\left(z-\frac12\right)u+\chi(u).
 \label{eq:Lz}
\end{equation}
If $z>1/2$, then $L_z(u)$ is strictly increasing in $u$.

\begin{lemma}\label{lem:fractional-exclusion}
For a finite weighted maximizer, every pair satisfies
\begin{equation}
 0<W_{ij}<1
 \quad\Longrightarrow\quad
 d_i\le\frac12\ \text{and}\ d_j\le\frac12.
 \label{eq:fractional-low}
\end{equation}
For the chosen extremal graphon, the same implication holds almost
everywhere.
\end{lemma}

\begin{proof}
Suppose first that $W$ is finite weighted. If a fractional pair has a
high-degree endpoint, let $X$ and $Y$ be the two corresponding weighted
vertex classes. Split either class into identical clones when needed. If
$X\ne Y$, then $0<W<1$ on $X\times Y$ and there are
$\rho,\theta>0$ such that
\[
 \rho\le W\le1-\rho\quad\text{on }X\times Y,
 \qquad d\ge\frac12+\theta\quad\text{on }Y.
\]
If $X=Y$, use the diagonal perturbation below.

For graphons, by symmetry it suffices to exclude the case in which
$E:=\{(x,y):d(y)>1/2,\ 0<W(x,y)<1\}$
has positive measure. For each fixed high degree $z$, the
equation $L_z(u)=0$
has at most one degree solution $u$. Hence, for almost every high-degree point
$y$, the section $E_y:=\{x:(x,y)\in E\}$ is contained modulo a null set in a
single degree fiber whenever $|E_y|>0$. Such a fiber has positive measure.
Since the positive-measure degree fibers are countable, Fubini gives a degree
atom $X=\{d=t\}$ of weight $|X|=m>0$, and a positive-measure set
$Y_0:=\{y:d(y)>1/2,\ 0<w(t,d(y))<1\}$.
By restricting to one member of the countable union
\[
 Y_0=\bigcup_{j,k\ge1}\left\{y\in Y_0:
 \frac1j\le w(t,d(y))\le1-\frac1j,\quad
 d(y)\ge\frac12+\frac1k\right\},
\]
we obtain $\rho,\theta>0$ and a positive-measure set
$Y\subseteq Y_0$ such that
\[
 \rho\le W\le1-\rho\quad\text{on }X\times Y,
 \qquad d\ge\frac12+\theta\quad\text{on }Y.
\]
If $Y\setminus X$ is null, then $X$ is a high-degree atom and the diagonal
argument below applies. Otherwise replace $Y$ by $Y\setminus X$, so that
$X\cap Y=\varnothing$.

Split $X=X_+\sqcup X_-$ into two sets of equal measure. Define a symmetric
perturbation $K$ to be $+1$ on $X_+\times Y$ and $Y\times X_+$, $-1$ on
$X_-\times Y$ and $Y\times X_-$, and $0$ elsewhere. Put $m=|X|$,
$n=|Y|$. Then
\begin{equation}
 k=n(\one_{X_+}-\one_{X_-}),
 \qquad \eta=0,
 \qquad \ip{k}{Wk}=0,
\label{eq:fractional-k}
\end{equation}
where the last equality uses degree measurability, which makes $W$ constant
on $X\times X$; in the finite case the same constancy follows from the
identical cloning of one vertex class. Furthermore,
\begin{equation}
 \ip{k}{Kd}=mn\int_Yd(y)\dd y,
 \qquad \|k\|_2^2=mn^2.
\label{eq:fractional-pairings}
\end{equation}
Thus \eqref{eq:Qform}, \eqref{eq:fractional-k}, and
\eqref{eq:fractional-pairings} give
\begin{equation}
 \mathcal Q_W(K)=mn\left(2\int_Yd(y)\dd y-n\right)>0.
 \label{eq:fiber-positive}
\end{equation}
The first-order term vanishes since $L=0$ on the support of $K$. Moreover,
$k$ is supported on $X$, whereas $Kk$ is supported on $Y$, and hence
$\ip{k}{Kk}=0$. For $0<|\varepsilon|<\rho$, the perturbation is feasible and
\eqref{eq:fiber-positive} shows that it strictly raises $\Phi$, a
contradiction.

In the diagonal case, $d=t>1/2$ on $X$ and $W$ is constant with
$0<W<1$ on $X\times X$. Split $X$ into equal halves and take
\begin{equation}
 K=\one_{X_+\times X_+}-\one_{X_-\times X_-}.
\label{eq:diagonal-perturbation}
\end{equation}
For the perturbation \eqref{eq:diagonal-perturbation},
\begin{equation}
 \eta=0,\qquad \ip{k}{Wk}=0,\qquad
 \ip{k}{Kd}=t\|k\|_2^2,
\label{eq:diagonal-pairings}
\end{equation}
and therefore \eqref{eq:Qform} and \eqref{eq:diagonal-pairings} give
\begin{equation}
 \mathcal Q_W(K)=(2t-1)\|k\|_2^2>0.
\label{eq:diagonal-positive}
\end{equation}
The first-order term vanishes since $L=0$ on $X\times X$. The function
$Kk$ takes the same value on $X_+$ and $X_-$, while $k$ takes opposite values
on two sets of equal measure; thus $\ip{k}{Kk}=0$. This is again a
contradiction by \eqref{eq:diagonal-positive}.
\end{proof}

\begin{lemma}\label{lem:upper-set}
In a finite degree-measurable maximizer, for every degree class $z>1/2$ the red
neighborhood $N_z:=\{j:w(z,d_j)=1\}$
is a degree-saturated upper set of total vertex weight $z$. In the compact
maximizer, for $\mu$-almost every degree value $z>1/2$, the corresponding
neighborhood $N_z:=\{y:w(z,d(y))=1\}$
is a degree-saturated upper set and
\begin{equation}
 |N_z|=z>\frac12.
 \label{eq:Nz-mass}
\end{equation}
In the finite model, $1/2<r<u$ implies $N_r\subseteq N_u$; in the compact
model this inclusion holds for $\mu^2$-almost every such pair, modulo null
sets.
\end{lemma}

\begin{proof}
By Lemma~\ref{lem:fractional-exclusion}, for $\mu$-almost every high-degree
value $z$ the corresponding row is $0/1$-valued modulo a null set. Fix such
a $z$ for which the rowwise KKT conditions also hold, and put
$E_z:=\{t\in[0,1]:w(z,t)=1\}$.
Outside a $\mu$-null set, membership in $E_z$ gives $L_z(t)\ge0$, while
nonmembership gives $L_z(t)\le0$. Consequently
\begin{equation}
 \mu^2\{(u,v):u<v,\ u\in E_z,\ v\notin E_z\}=0:
\label{eq:row-no-inversion}
\end{equation}
otherwise a good pair in this set would satisfy
$L_z(u)\ge0\ge L_z(v)$, contradicting the strict increase of $L_z$ in
\eqref{eq:Lz}. The no-inversion upper-ray construction at the end of
the proof of Lemma~\ref{lem:half-slope}, applied to
\eqref{eq:row-no-inversion}, therefore gives an upper-ray
representative of $E_z$. Its pullback under $d$ is degree-saturated, and its
weight is
\[
 \mu(E_z)=\int_{[0,1]}w(z,t)\,\dd\mu(t)=z.
\]
This proves \eqref{eq:Nz-mass}.
Upper rays are linearly ordered by inclusion. Hence, if $1/2<r<u$, their
weights $r<u$ force $N_r\subseteq N_u$, proving the asserted nesting
(for all good degree values, and therefore for $\mu^2$-almost every ordered
pair in the graphon case).
\end{proof}

The identity $(d,d)_\#(\lambda^2)=\mu^2$ transfers the almost-everywhere KKT
and fractional-edge conclusions to the degree square. Integrating
$W(x,y)=w(d(x),d(y))$ in $y$ and pushing forward by $d$ gives
\begin{equation}
 \int_{[0,1]}w(z,t)\,\dd\mu(t)=z.
\label{eq:row-degree-pushforward}
\end{equation}
Similarly, \eqref{eq:B-def} and \eqref{eq:degree-measurable} give
\begin{equation}
 B(z)=\int_{[0,1]}w(z,t)(1-t)\,\dd\mu(t).
\label{eq:B-row-pushforward}
\end{equation}
Since the exceptional row masses are Borel parameterized integrals,
Fubini's theorem and Lemma~\ref{lem:upper-set} give a Borel set
\begin{equation}
 D_\circ\subseteq[0,1],
 \qquad \mu(D_\circ)=1,
 \label{eq:common-degree-set}
\end{equation}
such that the KKT signs, the high-degree $0/1$ conclusion, and
\eqref{eq:row-degree-pushforward}--\eqref{eq:B-row-pushforward} hold on
each row indexed by $D_\circ$, outside a $\mu$-null set in the second
variable. For $z\in D_\circ\cap(1/2,1]$, let $N_z$ be the
degree-saturated upper-ray representative of that row, so $|N_z|=z$.
Define
\begin{equation}
 \mathcal R_\circ:=\left\{(r,u)\in D_\circ^2:
 \begin{array}{l}
  1/2<r<u,\quad B(u)-B(r)\le (u-r)/2,\\[-1mm]
  \displaystyle
  \int_{[0,1]}\one_{\{w(r,t)=1\}}
  \one_{\{w(u,t)\ne1\}}\,\dd\mu(t)=0
 \end{array}\right\}.
 \label{eq:common-pair-set}
\end{equation}
This set is measurable since the parameterized integral is Borel.
Lemmas~\ref{lem:upper-set} and \ref{lem:half-slope} show that it has full
$\mu^2$-measure relative to
$\{(r,u)\in D_\circ^2:1/2<r<u\}$. Below, degree values and ordered pairs are
chosen from $D_\circ$ and $\mathcal R_\circ$, respectively; Fubini's theorem
handles the corresponding null sections when row statements are pulled back
to the graphon.

\subsection{High-degree atomization}\label{sec:atomization}

Let $\mathcal V$ denote $V$ in the finite weighted model and $[0,1]$ in
the graphon model. Identical cloning permits subsets of any prescribed weight
inside a finite vertex class without changing the graph parameters.

Let $H:=\{v\in\mathcal V:d(v)>1/2\}$ and $h:=|H|$. Choose a degree
median $c$ satisfying
$|\{d>c\}|\le1/2\le|\{d\ge c\}|$. Select a subset
$T_0\subseteq\{d=c\}$ of weight $|T_0|=1/2-|\{d>c\}|$, and define
\begin{equation}
 T:=\{d>c\}\cup T_0,
 \qquad |T|=\frac12.
 \label{eq:median-core}
\end{equation}
The set $T$ may contain only part of the median degree fiber.

\begin{lemma}\label{lem:median-core}
If $U$ is a degree-saturated upper set of weight greater than $1/2$, then
\begin{equation}
 T\subseteq U,
 \label{eq:T-in-U}
\end{equation}
where inclusions are understood modulo null sets in the graphon case.
\end{lemma}

\begin{proof}
The set $U$ is an upper ray in the degree order together with whole degree
classes. If it omits a positive-weight
part of a class $\{d=t\}$ with $t>c$, saturation forces
$U\subseteq\{d>t\}$. Hence $|U|\le|\{d>c\}|\le1/2$, contrary to the
hypothesis.
Thus every degree class above $c$ is contained in $U$. If $U$ omitted a
positive-weight part of $T_0\subseteq\{d=c\}$, saturation would force it to
omit the entire class $\{d=c\}$. Then again $U\subseteq\{d>c\}$ and
$|U|\le1/2$, a contradiction.
\end{proof}

By Lemmas~\ref{lem:upper-set} and~\ref{lem:median-core}, and with the common
representatives fixed in \eqref{eq:common-degree-set}--\eqref{eq:common-pair-set},
applying \eqref{eq:T-in-U} to the neighborhoods in
\eqref{eq:Nz-mass} gives
\[
 T\subseteq N_z\quad\text{modulo null sets}
 \qquad\text{for every }z\in D_\circ\cap(1/2,1].
\]
Equivalently, $W=1$ almost everywhere on $H\times T$.

\begin{theorem}\label{thm:atomization}
For the chosen extremal graphon $W$ the following statements
hold.
\begin{enumerate}[label=(\roman*),leftmargin=2.5em]
 \item If $0<h\le1/2$, then $H$ is a clique and $d$ is almost everywhere
 constant on $H$.
 \item If $h>1/2$, then there is a highest-degree fiber $A=\{d=M\}$ of
 weight $|A|=a\ge1/2$ such that $A$ is a clique and $A$ is completely
 joined to all of $H$.
\end{enumerate}
\end{theorem}

\begin{proof}
Suppose first that $0<h\le1/2$. The top-half set $T$ may be chosen to contain
$H$: include all of $\{d>1/2\}$ first, then fill the remaining weight
$1/2-h$ from $H^c$ in decreasing degree order, splitting only the terminal
degree fiber. This is a median core of the form \eqref{eq:median-core}.
Every fixed high-degree row
$N_z$, $z\in D_\circ\cap(1/2,1]$, contains $T$ and
hence contains $H$ modulo null sets. Therefore $W=1$ on $H\times H$ modulo
a null set.

Let $\nu_H:=d_\#(\lambda|_H)$.
This finite measure is absolutely continuous with respect to $\mu$. Hence,
for $\nu_H^2$-almost every $r<u$, the pair $(r,u)$ belongs to
$\mathcal R_\circ$. For such a pair,
$N_r\subseteq N_u$ modulo null sets.
Let
$D=N_u\setminus N_r$. Both neighborhoods contain $H$, so $D\subseteq H^c$
and $d\le1/2$ on $D$. Since
\begin{equation}
 |D|=u-r,
 \qquad B(u)-B(r)=\int_D(1-d(x))\dd x.
\label{eq:difference-band}
\end{equation}
Equation \eqref{eq:difference-band} gives
$B(u)-B(r)\ge(u-r)/2$.
Together with Lemma~\ref{lem:half-slope}, this forces equality and hence
$\chi(u)=\chi(r)$ for $\nu_H^2$-almost every $r<u$.
Lemma~\ref{lem:ordered-constancy} now shows that $\chi$ is constant
$\nu_H$-almost everywhere. Thus $\chi(d(x))$ is almost everywhere constant
on $H$; denote this constant by $\chi_H$.

Let $\mathcal L$ denote the Lebesgue $\sigma$-algebra and set
\[
 \mathcal G_H:=
 \{E\in\mathcal L:E\cap H\in\{\varnothing,H\}\text{ modulo null sets}\}.
\]
Thus $\mathcal G_H$ retains every measurable subset of $H^c$ and collapses
$H$ to one atom. Let $P_H=\mathbb E[\cdot\mid\mathcal G_H]$ and use the
quotient $W^{P_H}$ from \eqref{eq:general-quotient}. Then
\[
 z=d-P_Hd
 =\left(d-\frac1{|H|}\int_Hd(x)\dd x\right)\one_H.
\]
Since $H$ is a clique and $\int_Hz(x)\dd x=0$,
$\ip{z}{Wz}=0$. Since $B=d/2-\chi_H$ on $H$,
$2\ip{z}{B}=\|z\|_2^2$. Equation \eqref{eq:quotient-identity} therefore
gives no loss in $\Phi$, and \eqref{eq:strict-quotient} forces $z=0$,
proving (i).

Now assume $h>1/2$. Choose the median core $T$ so that $T\subseteq H$.
Every fixed high-degree row
$N_z$, $z\in D_\circ\cap(1/2,1]$, contains $T$ modulo null
sets,
so $T$ is a clique; by symmetry, $T$ is completely joined to $H$. Define the
restricted degree distribution $\nu_T:=d_\#(\lambda|_T)$.
Again $\nu_T\ll\mu$. Using \eqref{eq:degree-measurable}, the pushforward
definition of $\nu_T$, and $W=1$ almost everywhere on $T\times H$, we have
\begin{equation}
 0
 =\int_T\int_H\bigl(1-W(x,y)\bigr)\,\dd y\,\dd x
 =\int_{[0,1]}
   \left[\int_H\bigl(1-w(r,d(y))\bigr)\,\dd y\right]\dd\nu_T(r).
\label{eq:T-H-nonedge-integral}
\end{equation}
The integrand in brackets is nonnegative, so
\eqref{eq:T-H-nonedge-integral} shows that it vanishes for
$\nu_T$-almost every $r$; for every such $r\in D_\circ$, the fixed
representative $N_r$ contains $H$ modulo null sets. Consequently, for
$\nu_T^2$-almost every $r<u$, the pair lies in $\mathcal R_\circ$, both rows
contain $H$, and the same difference-band argument gives
$\chi(r)=\chi(u)$. Lemma~\ref{lem:ordered-constancy} now shows that
$\chi(d(x))$ is almost everywhere constant on $T$; denote the constant by
$\chi_T$. Define
\[
 \mathcal G_T:=
 \{E\in\mathcal L:E\cap T\in\{\varnothing,T\}\text{ modulo null sets}\},
 \qquad
 P_T=\mathbb E[\cdot\mid\mathcal G_T].
\]
Thus $\mathcal G_T$ collapses precisely $T$ and retains the full measurable
structure on $T^c$; use the quotient $W^{P_T}$ from
\eqref{eq:general-quotient}. If
\[
 z=d-P_Td
 =\left(d-\frac1{|T|}\int_Td(x)\dd x\right)\one_T,
\]
then $T$ is a clique and $\int_Tz(x)\dd x=0$, so $\ip{z}{Wz}=0$. Since
$B=d/2-\chi_T$ on $T$, we also have
$2\ip{z}{B}=\|z\|_2^2$. Equation \eqref{eq:quotient-identity} therefore
shows that merging $T$ leaves $\Phi$ unchanged, and
\eqref{eq:strict-quotient} implies that $d=M$ almost everywhere on $T$.
Let $A=\{d=M\}$ be the full degree fiber. Then $|A|\ge|T|=1/2$.
Since $T$ is a top-half set, $M$ is the essential maximum degree.
Degree measurability and $d=M$ on $T$ give
\begin{equation}
 0=\int_{T\times T}(1-W(x,y))\,\dd x\,\dd y
   =|T|^2\bigl(1-w(M,M)\bigr).
\label{eq:T-clique-lift}
\end{equation}
Thus \eqref{eq:T-clique-lift} gives $w(M,M)=1$, and therefore $W=1$ almost
everywhere on $A\times A$.
Likewise,
\begin{equation}
 0=\int_{T\times H}(1-W(x,y))\,\dd x\,\dd y
   =|T|\int_H\bigl(1-w(M,d(y))\bigr)\,\dd y.
\label{eq:T-H-lift}
\end{equation}
By \eqref{eq:T-H-lift}, $w(M,d(y))=1$ for almost every $y\in H$; hence
$W=1$ almost everywhere on $A\times H$.
\end{proof}

\begin{lemma}\label{lem:bernstein-sign}
For $0\le k\le n$, let
$B_{k,n}(t):=\binom nk t^k(1-t)^{n-k}$. If a polynomial $q$ of degree at
most $n$ has the exact Bernstein expansion
$q(t)=\sum_{k=0}^n c_kB_{k,n}(t)$
with $c_k\le0$ for every $k$, then $q(t)\le0$ for all $t\in[0,1]$. If all
coefficients are strictly negative, then $q(t)<0$ throughout $[0,1]$.
\end{lemma}

\begin{proof}
For $0\le t\le1$, every $B_{k,n}(t)$ is nonnegative and
$\sum_{k=0}^nB_{k,n}(t)=(t+(1-t))^n=1$.
Thus $q(t)$ is a convex combination of the coefficients
$c_0,\ldots,c_n$, which proves both assertions.
\end{proof}

\subsection{The small-atom estimate}\label{sec:small-atom}

Assume $0<h\le1/2$, and let $A:=H$ be the clique degree atom supplied by
Theorem~\ref{thm:atomization}(i). Write $a:=|A|$, so
$0<a\le1/2$, and let $M>1/2$ be the common degree on $A$. In particular,
$A$ is the entire high-degree set and every vertex outside $A$ has degree at
most $1/2$. By degree measurability and
Lemma~\ref{lem:fractional-exclusion}, almost all rows with first endpoint in
$A$ have the same $0/1$ red neighborhood; denote this common neighborhood,
defined modulo null sets, by $N_A$. Write $S=N_A\setminus A$, with
$|S|=M-a$, and $L=\mathcal V\setminus N_A$, with $|L|=1-M$.
Vertices in $S$ are adjacent to all of $A$, so their degrees are at least
$a$. Since $M>a$, put $u=(M-a)^{-1}\int_Sd(y)\dd y$. If $M<1$, put
$x=(1-M)^{-1}\int_Ld(y)\dd y$;
if $M=1$, then $L$ is null and we set $x=0$. Every formula below contains
$x$ only through a factor $1-M$ in this boundary case. Then
\[
 a\le u\le\frac12,
 \qquad0\le x\le\frac12.
\]
Define
\begin{equation}
 h_{a,M}(t):=\gfun(t)+a(M-t)^2.
 \label{eq:h-small}
\end{equation}
The degree integral and the forced $A$--$S$ part of the Dirichlet energy are
\begin{align}
 \beta&=aM+\int_Sd+\int_Ld
       =aM+(M-a)u+(1-M)x,
 \label{eq:small-beta}\\
 \Ecal_W(d)&\ge
 \frac12\int_{A\times S}(M-d(y))^2\dd(x,y)
 +\frac12\int_{S\times A}(d(x)-M)^2\dd(x,y)
 =a\int_S(M-d(y))^2\dd y.
 \label{eq:small-energy}
\end{align}
Consequently, the total contribution of $S$ to
$\int\gfun(d)+\Ecal_W(d)$ is at least
\begin{equation}
 \int_Sh_{a,M}(d(y))\dd y
 \ge(M-a)\underline h_{a,M}(u),
\label{eq:small-S-loss}
\end{equation}
while
\begin{equation}
 \int_L\gfun(d(y))\dd y
 \ge(1-M)\gcvx(x).
\label{eq:small-L-loss}
\end{equation}
Here both inequalities follow from Theorem~\ref{thm:jensen} applied to the
corresponding lower convex envelopes, as in
\eqref{eq:envelope-jensen}. Substituting \eqref{eq:small-beta},
\eqref{eq:small-energy}, \eqref{eq:small-S-loss}, and
\eqref{eq:small-L-loss} into
\eqref{eq:dirichlet} gives
\begin{equation}
 \Phi(W)\le U(a,M,u,x),
 \label{eq:U-bound}
\end{equation}
where
\begin{align}
 U(a,M,u,x)
 ={}&\beta(1-\beta)-a\gfun(M)
 -(M-a)\underline h_{a,M}(u)
 -(1-M)\gcvx(x),
 \label{eq:U-def}\\
 \beta={}&aM+(M-a)u+(1-M)x.
 \label{eq:beta-small}
\end{align}

The lower convex envelope of $h_{a,M}$ on $[a,1/2]$ has a simple form. Assume
$0\le a<1/2$, and put
\begin{equation}
 c(a):=\frac{1+2a}{4}.
 \label{eq:c-a}
\end{equation}
Let $\ell_{a,M}$ be the line through
$(c(a),h_{a,M}(c(a)))$ and $(1/2,h_{a,M}(1/2))$. Then
\begin{equation}
 h_{a,M}(t)-\ell_{a,M}(t)
 =\frac{(1-2t)(4t-1-2a)^2}{32}.
 \label{eq:h-envelope-cert}
\end{equation}
Moreover,
\begin{equation}
 h_{a,M}''(t)=2(1+a-3t)\ge\frac{1-2a}{2}\ge0
 \qquad(a\le t\le c(a)).
\label{eq:h-envelope-convexity}
\end{equation}
The double zero in \eqref{eq:h-envelope-cert} at $t=c(a)$ shows that the
affine piece is tangent there, while \eqref{eq:h-envelope-convexity} gives
convexity of the first piece. The other zero in
\eqref{eq:h-envelope-cert} is at $t=1/2$. Hence the two pieces form a convex
minorant. Any convex minorant must lie below the chord
joining its values at $c(a)$ and $1/2$, and therefore below $\ell_{a,M}$ on
that interval. Consequently, for $a<1/2$,
\begin{equation}
 \underline h_{a,M}(t)=
 \begin{cases}
  h_{a,M}(t),&a\le t\le c(a),\\[1mm]
  \ell_{a,M}(t),&c(a)\le t\le1/2.
 \end{cases}
 \label{eq:h-envelope}
\end{equation}
For $a=1/2$, \eqref{eq:h-envelope} is understood in the preceding singleton
sense; only its first branch is present.

\begin{lemma}\label{lem:small-atom}
For
\[
 0<a\le\frac12,\qquad \frac12<M\le1,\qquad
 a\le u\le\frac12,\qquad 0\le x\le\frac12,
\]
the function $U$ defined in \eqref{eq:U-def}--\eqref{eq:beta-small} satisfies
\begin{equation}
 U(a,M,u,x)\le\frac3{20}.
 \label{eq:small-atom-bound}
\end{equation}
\end{lemma}

\begin{proof}
First suppose $u\le c(a)$, so $\underline h_{a,M}(u)=h_{a,M}(u)$. Define
\begin{equation}
 \beta_0=Mu+(1-M)x,\qquad
 U_0=\beta_0(1-\beta_0)-M\gfun(u)-(1-M)\gcvx(x).
\label{eq:small-U0-def}
\end{equation}
By $\gfun\ge\gcvx$ and Theorem~\ref{thm:jensen} applied to $\gcvx$,
\begin{equation}
 U_0\le \beta_0(1-\beta_0)-\gcvx(\beta_0)
 \le\frac{153}{1024}.
 \label{eq:U0}
\end{equation}
Here $0\le \beta_0\le1/2$. On $[0,1/4]$ the last function is
$t(1-t)^2\le9/64$, while on $[1/4,1/2]$ it is
$-t^2+11t/16+1/32$, whose maximum is $153/1024$ at $t=11/32$.
Put $\Delta=a(M-u)$, so that $\beta=\beta_0+\Delta$. Then
\begin{align}
 \beta(1-\beta)-\beta_0(1-\beta_0)
 &=\Delta(1-2\beta_0)-\Delta^2,\notag\\
 \frac{\gfun(M)-\gfun(u)}{M-u}
 &=M+u-M^2-Mu-u^2,\notag\\
 a(M-a)(M-u)^2+a^2(M-u)^2
 &=aM(M-u)^2.
\label{eq:small-affine-identities}
\end{align}
Using \eqref{eq:small-beta}, \eqref{eq:small-U0-def}, and
\eqref{eq:small-affine-identities}, we obtain
\begin{equation}
 U=U_0+a(M-u)R,
 \qquad
 R=1-M-u+u^2-2(1-M)x.
 \label{eq:U-affine}
\end{equation}
If $R\le0$, we are done. If $R>0$, then for fixed $M,u,x$ the expression in \eqref{eq:U-affine} is increasing in $a$. Since the feasible range has $a\le u$, we may increase $a$ to $u$. This stays in the same branch since
$u\le c(u)=(1+2u)/4$ if and only if $u\le1/2$.
Put $V(a)=U(a,M,a,x)$ and
$R_a=1-M-a+a^2-2(1-M)x$.
Then
\begin{equation}
 V'(a)=(M-a)\bigl(2R_a+(1-2a)(M-a)\bigr).
 \label{eq:V-prime}
\end{equation}
As long as $R_a>0$, \eqref{eq:V-prime} is positive. Moreover,
$R_a'=2a-1\le0$ on $[0,1/2]$. Thus, once $R_a\le0$, it remains
nonpositive and \eqref{eq:U-affine} together with \eqref{eq:U0} bounds all
subsequent values by $153/1024$. If $R_a$ remains positive, then $V$ is
increasing and its maximum occurs at $a=u=1/2$. Hence
\begin{equation}
 U\le\max\left\{\frac{153}{1024},F_{\mathrm{aux}}(M,x)\right\},
 \qquad
 F_{\mathrm{aux}}(M,x):=U\left(\frac12,M,\frac12,x\right).
\label{eq:F-reduction}
\end{equation}
By \eqref{eq:F-reduction}, it remains to bound $F_{\mathrm{aux}}$.

For $x\ge1/4$,
\begin{equation}
 \frac{\partial F_{\mathrm{aux}}}{\partial x}
 =(1-M)\left(\frac{11}{16}-2\beta\right)\le0,
 \qquad \beta=M-\frac14+(1-M)x.
\label{eq:F-aux-x-derivative}
\end{equation}
Here $\beta\ge3/8>11/32$, and equality in $\beta\ge3/8$ occurs only at
$M=1/2,x=1/4$. Hence \eqref{eq:F-aux-x-derivative} is strictly negative for
$M<1$, so the maximum is attained at $x=1/4$.
There
\begin{equation}
 \frac3{20}-F_{\mathrm{aux}}\left(M,\frac14\right)
 =\frac{100M^2-95M+23}{320}>0.
\label{eq:F-high-x}
\end{equation}
Thus \eqref{eq:F-high-x} handles the region $x\ge1/4$. The numerator is
positive for every real $M$, since its leading
coefficient is positive and its discriminant is
$95^2-4\cdot100\cdot23=-175<0$.
For $0\le x\le1/4$, $F_{\mathrm{aux}}$ is a strictly concave quadratic in
$M$, since direct differentiation gives
\begin{equation}
 \frac{\partial^2F_{\mathrm{aux}}}{\partial M^2}
 =-\frac{(2x-3)(2x-1)}2<0.
\label{eq:F-aux-concavity}
\end{equation}
By \eqref{eq:F-aux-concavity}, its unique stationary point is
\begin{equation}
 M_{\mathrm v}(x)
 =-\frac{2x^3-6x^2+7x-2}{(2x-3)(2x-1)}.
\label{eq:F-aux-stationary}
\end{equation}
Substituting \eqref{eq:F-aux-stationary} into $F_{\mathrm{aux}}$ gives the
unrestricted maximum
\begin{equation}
 F_{\mathrm{aux},\ast}(x)=
 \frac{16x^6-32x^5+32x^3-8x^2-16x+7}
 {16(2x-3)(2x-1)}.
\label{eq:F-star}
\end{equation}
Moreover, \eqref{eq:F-star} gives
\begin{equation}
 \frac3{20}-F_{\mathrm{aux},\ast}(x)
 =-\frac{P(x)}{80(2x-3)(2x-1)},
 \label{eq:P-gap}
\end{equation}
where
\begin{equation}
 P(x)=80x^6-160x^5+160x^3-88x^2+16x-1.
 \label{eq:P-small}
\end{equation}
With $t=4x\in[0,1]$ and
$B_{k,7}(t)=\binom7k t^k(1-t)^{7-k}$,
\begin{equation}
 P(t/4)=-B_{0,7}-\frac37B_{1,7}-\frac5{42}B_{2,7}
 -\frac{37}{672}B_{5,7}-\frac{29}{256}B_{6,7}
 -\frac{35}{256}B_{7,7}.
 \label{eq:P-Bernstein}
\end{equation}
Expanding \eqref{eq:P-Bernstein} verifies the identity with
\eqref{eq:P-small}. The coefficients of $B_{3,7}$ and $B_{4,7}$ are $0$,
and all remaining coefficients are negative. Hence
Lemma~\ref{lem:bernstein-sign} gives $P\le0$. Since
$(2x-3)(2x-1)>0$ for $0\le x\le1/4$, \eqref{eq:P-gap} therefore gives
$F_{\mathrm{aux}}\le3/20$.

It remains to consider $u>c(a)$; the equality case was included in the first
branch. In particular $a<1/2$. Put
\begin{equation}
 e=\frac12-a,\qquad m=M-\frac12,\qquad
 \tau=(M-a)\left(\frac12-u\right).
\label{eq:affine-parameters}
\end{equation}
Then $0\le \tau\le e(m+e)/2$. If
$R_0=3/4-M-2(1-M)x$,
then
\begin{equation}
 a=\frac12-e,\qquad M=\frac12+m,\qquad
 M-a=m+e,\qquad u=\frac12-\frac{\tau}{m+e}.
\label{eq:affine-substitution}
\end{equation}
On the present branch, the second line of \eqref{eq:h-envelope} is
\begin{equation}
 \underline h_{a,M}(u)
 =\frac{\frac12-u}{\frac12-c(a)}h_{a,M}(c(a))
  +\frac{u-c(a)}{\frac12-c(a)}h_{a,M}\left(\frac12\right).
\label{eq:h-affine-branch}
\end{equation}
Substituting \eqref{eq:affine-parameters}, \eqref{eq:affine-substitution},
and \eqref{eq:h-affine-branch}, and using \eqref{eq:c-a}, into
\eqref{eq:U-def}--\eqref{eq:beta-small} gives,
after collecting the terms in $e$ and $\tau$,
\begin{equation}
 U=F_{\mathrm{aux}}(M,x)-emR_0
 +\tau\left(\frac{e^2}{4}-R_0-\tau\right).
 \label{eq:U-convex-e}
\end{equation}
For fixed $M,x,\tau$, the right-hand side is convex in $e$, with second
derivative $\tau/2$. Since $M-a=m+e$ and
\eqref{eq:c-a} and \eqref{eq:affine-parameters} give
$c(a)=1/2-e/2$,
after adjoining the already-treated boundary $u=c(a)$, the closed condition
$u\ge c(a)$ is equivalent to
\[
 \frac12-u\le\frac e2,
 \qquad\text{or}\qquad
 \tau\le \frac{e(m+e)}2.
\]
Together with $0\le e\le1/2$, this gives a closed interval after adjoining the
boundary value $a=0$. Since $m>0$, the function
$e\mapsto e(m+e)$ is strictly increasing on $[0,1/2]$, so for fixed feasible
$\tau$ the allowed interval is
\[
 \left[e_{\min},\frac12\right],
 \qquad
 e_{\min}=\frac{\sqrt{m^2+8\tau}-m}{2}.
\]
Maximizing on this closure can only enlarge the original supremum. The
right-hand side of \eqref{eq:U-convex-e} therefore attains its maximum at an
endpoint. At the lower endpoint,
$\tau=e_{\min}(m+e_{\min})/2$, which is exactly $u=c(a)$ and has already
been treated. The upper endpoint is $e=1/2$, or $a=0$. At this endpoint the bound becomes a two-point
Jensen bound: here $\underline h_{0,M}=\gcvx$ on $[0,1/2]$, and
Theorem~\ref{thm:jensen} gives
\[
 U\le \beta(1-\beta)-\gcvx(\beta)\le\frac{153}{1024}.
\]
\end{proof}

\subsection{The majority estimate}\label{sec:majority-clique}

Suppose that the chosen extremal graphon has a highest-degree clique atom
$A$ of weight $a\ge1/2$ and degree $M$, completely joined to the high-degree
set. By degree measurability and
Lemma~\ref{lem:fractional-exclusion}, its rows share a common $0/1$ red
neighborhood modulo null sets; denote it by $N_A$. Write
$S=N_A\setminus A$, with $|S|=M-a$, and
$L=\mathcal V\setminus N_A$, with $|L|=1-M$.
For $y\in S$ we have $a\le d(y)\le M$, while $d(y)\le1/2$ on $L$. If
$M=a$, then $S$ is empty. Assume $M>a$. By \eqref{eq:h-small}, we have
\begin{equation}
 h_{a,M}''(t)=2(1+a-3t)\le2(1-2a)\le0
 \quad(a\le t\le M),
\label{eq:majority-concavity}
\end{equation}
so \eqref{eq:majority-concavity} shows that $h_{a,M}$ is concave on
$[a,M]$. Set
\begin{equation}
 z=\int_S\frac{M-d(y)}{M-a}\dd y,
 \qquad0\le z\le M-a.
 \label{eq:z-def}
\end{equation}
If $M<1$, set $x=(1-M)^{-1}\int_Ld(y)\dd y$.
If $M=1$, then $L$ is null; in that case take any $x\in[0,1/2]$, since every
subsequent expression is independent of $x$ after the factor $1-M$ is
inserted. For $t\in[a,M]$, concavity of $h_{a,M}$ gives
the endpoint-chord estimate
\begin{equation}
 h_{a,M}(t)\ge
 \frac{M-t}{M-a}h_{a,M}(a)
 +\frac{t-a}{M-a}h_{a,M}(M).
 \label{eq:majority-chord}
\end{equation}
The definition of $z$ is equivalent to
\begin{equation}
 \int_S(M-d(y))\dd y=z(M-a).
 \label{eq:z-moment}
\end{equation}
Thus integration of \eqref{eq:majority-chord} over $S$ assigns mass $z$ to
the endpoint $a$ and mass $M-a-z$ to the endpoint $M$. The endpoint values are
\begin{equation}
 h_{a,M}(a)=\gfun(a)+a(M-a)^2,\qquad
 h_{a,M}(M)=\gfun(M).
\label{eq:majority-endpoints}
\end{equation}
Combining \eqref{eq:majority-chord}, \eqref{eq:z-moment}, and
\eqref{eq:majority-endpoints}, then adding the contribution $a\gfun(M)$
from $A$, retaining the forced $A$--$S$ energy, and applying $\gcvx$ on $L$
gives the loss term below.
Also, \eqref{eq:z-moment} gives
\begin{equation}
 \int_Sd=M(M-a)-z(M-a).
\label{eq:majority-S-degree}
\end{equation}
Substituting \eqref{eq:majority-S-degree} into the degree integral gives
\begin{equation}
 \beta=aM+\int_Sd+(1-M)x
 =M^2-z(M-a)+(1-M)x.
\label{eq:majority-beta-computation}
\end{equation}
Thus \eqref{eq:majority-beta-computation} and the preceding loss estimate,
inserted into \eqref{eq:dirichlet}, give
\begin{align}
 \Phi(W)\le{}&\beta(1-\beta)-(M-z)\gfun(M)-z\gfun(a)
 -az(M-a)^2\quad-(1-M)\gcvx(x).
 \label{eq:majority-pre}
\end{align}
To obtain the next form, put
\begin{equation}
 \beta_0=M^2+(1-M)x,\qquad \delta=(M-a)z.
\label{eq:majority-beta0}
\end{equation}
Then $\beta=\beta_0-\delta$ and
\begin{equation}
 \beta(1-\beta)=\beta_0(1-\beta_0)
 +\delta(2\beta_0-1)-\delta^2.
\label{eq:majority-beta-expansion}
\end{equation}
Moreover,
\begin{align}
 &2\beta_0-1+
 \frac{\gfun(M)-\gfun(a)-a(M-a)^2}{M-a}=-\bigl((1-M)(1-2x)+a(2M-1)-M^2\bigr)
 =-\Hcal(a,M,x).
\label{eq:majority-H-identity}
\end{align}
Substituting \eqref{eq:majority-beta0},
\eqref{eq:majority-beta-expansion}, and
\eqref{eq:majority-H-identity} into \eqref{eq:majority-pre} gives
\begin{equation}
 \Phi(W)\le F_0(M,x)-(M-a)^2z^2-(M-a)\Hcal(a,M,x)z,
 \label{eq:majority-master}
\end{equation}
where
\begin{equation}
 \Hcal(a,M,x)=(1-M)(1-2x)+a(2M-1)-M^2.
 \label{eq:Hcal}
\end{equation}

\begin{lemma}\label{lem:negative-H}
Let
\[
 \frac12\le a<M\le1,\qquad 0\le x\le\frac12,
 \qquad 0\le z\le M-a.
\]
If $\Hcal(a,M,x)<0$, then
\begin{equation}
 F_0(M,x)-(M-a)^2z^2-(M-a)\Hcal(a,M,x)z
 \le\frac3{20}.
\label{eq:negative-H-bound}
\end{equation}
Consequently, \eqref{eq:majority-master} and
\eqref{eq:negative-H-bound} imply $\Phi(W)\le3/20$.
\end{lemma}

\begin{proof}
For $y\in\mathbb R$, write $y^-:=\max\{-y,0\}$. Put
\begin{equation}
 \Hcal_0(M,x)=\Hcal\left(\frac12,M,x\right)
 =\frac12-M^2-2(1-M)x.
\label{eq:H0}
\end{equation}
By \eqref{eq:Hcal} and \eqref{eq:H0}, since $a\ge1/2$ and $M\ge1/2$, we have
$\Hcal(a,M,x)\ge\Hcal_0(M,x)$. Moreover, $M>a$ in the present branch, and
maximizing the quadratic in $z$ without its interval constraint gives
\begin{equation}
 -(M-a)^2z^2-(M-a)\Hcal(a,M,x)z
 \le\frac{(\Hcal(a,M,x)^-)^2}{4}
 \le\frac{(\Hcal_0^-)^2}{4}.
\label{eq:negative-H-quadratic}
\end{equation}
Consequently, \eqref{eq:negative-H-quadratic} gives
\begin{equation}
 F_0(M,x)-(M-a)^2z^2-(M-a)\Hcal(a,M,x)z
 \le W_0(M,x):=F_0(M,x)+\frac{(\Hcal_0^-)^2}{4}.
 \label{eq:W0}
\end{equation}
It remains to prove $W_0\le3/20$ in the region $\Hcal_0\le0$.

If $x\ge1/4$, then
\begin{equation}
 \Hcal_0(M,x)\le\Hcal_0\left(M,\frac14\right)
 =M\left(\frac12-M\right)\le0.
\label{eq:H0-high-x}
\end{equation}
Thus \eqref{eq:H0-high-x} shows that the negative-part term in
\eqref{eq:W0} is active throughout this region.
Differentiation gives
\begin{equation}
 \frac{\partial W_0}{\partial x}
 =\frac{(M-1)(16M^2-3)}{16}\le0.
\label{eq:W0-high-derivative}
\end{equation}
Thus \eqref{eq:W0-high-derivative} shows that the maximum occurs at $x=1/4$,
and
\begin{equation}
 \frac{19}{128}-W_0\left(M,\frac14\right)
 =-\frac{(2M-1)(16M^3-40M^2+12M+1)}{128}\ge0.
\label{eq:W-high-x}
\end{equation}
To verify the sign in \eqref{eq:W-high-x}, observe that the cubic factor is
decreasing on $[1/2,1]$ and equals $-1$ at $M=1/2$. Indeed, its derivative
is given by \eqref{eq:W-high-cubic-derivative}:
\begin{equation}
 4(2M-3)(6M-1)<0
 \qquad(1/2\le M\le1).
\label{eq:W-high-cubic-derivative}
\end{equation}

Suppose $x\le1/4$. Then
\begin{equation}
 \frac{\partial W_0}{\partial x}
 =\frac{1-M}{2}(1-2M^2-4x+6x^2).
 \label{eq:W-low-derivative}
\end{equation}
If $M\le1/\sqrt2$, the condition $\Hcal_0\le0$ gives
$x\ge x_0(M):=(1/2-M^2)/(2(1-M))$.
On $1/2\le M\le1/\sqrt2$, the numerator and denominator in this expression
are nonnegative and positive, respectively, so $x_0(M)\ge0$. Moreover,
$x_0(M)\le1/4$ is equivalent to $M^2\ge M/2$, which holds since
$M\ge1/2$.
The bracket in \eqref{eq:W-low-derivative} is decreasing in $x$ on
$[0,1/4]$, and at $x=x_0(M)$ it equals
\begin{equation}
 -\frac{(2M^2-1)(2M^2-8M+3)}{8(M-1)^2}\le0.
\label{eq:W-low-at-x0}
\end{equation}
Thus \eqref{eq:W-low-derivative} and \eqref{eq:W-low-at-x0} show that the
derivative is nonpositive from $x_0(M)$ onward, so the maximum
occurs at $x=x_0(M)$. There
\begin{equation}
 \frac3{20}-W_0(M,x_0(M))
 =-\frac{P_6(M)}{320(M-1)^2},
 \label{eq:P6-gap}
\end{equation}
where
\begin{equation}
 P_6(M)=200M^6-720M^5+940M^4-560M^3+142M^2-4M-3.
 \label{eq:P6}
\end{equation}
After $M=1/2+t/4$, an exact change to the degree-$6$ Bernstein basis gives
\begin{align}
 P_6\left(\frac12+\frac t4\right)
 ={}&-\frac18B_{0,6}(t)-\frac5{48}B_{1,6}(t)
 -\frac{41}{480}B_{2,6}(t)-\frac{17}{320}B_{3,6}(t)-\frac{13}{640}B_{4,6}(t)-\frac{37}{768}B_{5,6}(t)
 -\frac{111}{512}B_{6,6}(t).
 \label{eq:P6-Bernstein}
\end{align}
Expanding \eqref{eq:P6-Bernstein} verifies the identity with
\eqref{eq:P6} after substituting $M=1/2+t/4$. All coefficients are strictly
negative. The range $1/2\le M\le1/\sqrt2$ corresponds to a subinterval of
$0\le t\le1$, so Lemma~\ref{lem:bernstein-sign} gives $P_6<0$ throughout
the required range. Since $(M-1)^2>0$ there, \eqref{eq:P6-gap} indeed gives
$3/20-W_0(M,x_0(M))>0$.

If $M\ge1/\sqrt2$, the same bracket is already nonpositive at $x=0$ and
decreases on $[0,1/4]$. Hence the maximum is at $x=0$. Finally,
\begin{equation}
 \frac3{20}-W_0(M,0)
 =-\frac{20M^4-80M^3+60M^2-7}{80}\ge0.
\label{eq:W-x0}
\end{equation}
To verify \eqref{eq:W-x0}, note that the quartic in the numerator is
decreasing on
$[1/\sqrt2,1]$: its derivative is given by
\eqref{eq:W-x0-quartic-derivative}:
\begin{equation}
 40M(2M^2-6M+3)<0
 \qquad(1/\sqrt2\le M\le1),
\label{eq:W-x0-quartic-derivative}
\end{equation}
since $2M^2-6M+3$ is decreasing on this interval and already equals
$4-3\sqrt2<0$ at its left endpoint. Its value at $1/\sqrt2$ is
$28-20\sqrt2<0$.
\end{proof}

\begin{theorem}\label{thm:high-collapse}
For the chosen extremal graphon, the essential degree support
in $(1/2,1]$ contains at most one value.
\end{theorem}

\begin{proof}
If $h=0$, the high-degree support is empty and the claim is trivial. If
$0<h\le1/2$, the claim is Theorem~\ref{thm:atomization}(i). Assume $h>1/2$ and
take the majority clique atom $A$ from Theorem~\ref{thm:atomization}(ii).

If $M=a$, then $S=N_A\setminus A$ is null. Since $A$ is completely joined to
$H$, every high-degree point lies in $N_A$, so $H=A$ modulo a null set and the
claim follows. Hence assume $M>a$ below.

If $\Hcal(a,M,x)<0$, Lemma~\ref{lem:negative-H} gives
$\Phi(W)\le3/20$. This is impossible for a maximizer since the rational
split-regular construction in \eqref{eq:rational-benchmark} has larger value.

If $\Hcal(a,M,x)\ge0$, then \eqref{eq:majority-master} and
\eqref{eq:F0-bound} give
$
 \Phi(W)\le F_0(M,x)\le\pstar.
$ 
The split-regular construction from Lemma~\ref{lem:regular-max} gives a lower bound
$\pstar$, so equality is required throughout. In particular, $z=0$. If
$M>a$, then \eqref{eq:z-def} implies $d=M$ almost everywhere on $S$,
contradicting the fact that $A$ is the full degree-$M$ fiber and
$S\cap A=\varnothing$. Therefore $S=\varnothing$ and $M=a$. Since $A$ was
completely joined to every high-degree point, it follows that $H=A$ modulo a
null set.
\end{proof}

\subsection{Proof of Theorem~\ref{thm:structural-reduction}}

\begin{proof}
Let $H=\{d>1/2\}$ and $h=|H|$. Suppose first that $h=0$. Then
$0\le d\le1/2$ and $0\le\beta\le1/2$. Lemma~\ref{lem:dirichlet},
$\Ecal_W(d)\ge0$, $\gfun\ge\gcvx$, and Theorem~\ref{thm:jensen} give
\begin{align}
 \Phi(W)
 &\le\beta(1-\beta)-\int_0^1\gfun(d(t))\dd t\notag\\
 &\le\beta(1-\beta)-\int_0^1\gcvx(d(t))\dd t\notag\\
 &\le\beta(1-\beta)-\gcvx(\beta)
 \le\frac{153}{1024}<\frac3{20}.
 \label{eq:no-high-bound}
\end{align}
For the last maximum, \eqref{eq:g-envelope} gives
\begin{equation}
 t(1-t)-\gcvx(t)=
 \begin{cases}
  t(1-t)^2,&0\le t\le1/4,\\[1mm]
  -t^2+\dfrac{11t}{16}+\dfrac1{32},&1/4\le t\le1/2.
 \end{cases}
\label{eq:no-high-envelope-gap}
\end{equation}
By \eqref{eq:no-high-envelope-gap}, the second quadratic has its maximum
$153/1024$ at $t=11/32$, while
the first branch is at most $9/64<153/1024$. This contradicts the lower
bound \eqref{eq:rational-benchmark} for a maximizer in view of
\eqref{eq:no-high-bound}.

If $0<h\le1/2$, Theorem~\ref{thm:atomization}(i) makes $H$ a single clique
degree atom of weight at most $1/2$. Equations \eqref{eq:U-bound} and
\eqref{eq:small-atom-bound} then give $\Phi(W)\le3/20$, again contradicting
\eqref{eq:rational-benchmark}. Consequently,
\begin{equation}
 h>\frac12.
\label{eq:high-majority}
\end{equation}

By \eqref{eq:high-majority} and Theorem~\ref{thm:high-collapse}, the
essential degree support on
$H$ consists of one value $M$. Thus $A:=H$ is the full degree-$M$ fiber,
and Theorem~\ref{thm:atomization}(ii) shows that $A$ is a clique. Put
$a=|A|=h$ and $B=[0,1]\setminus A$. Lemma~\ref{lem:fractional-exclusion}
and \eqref{eq:fractional-low}, together with
\eqref{eq:degree-measurable} imply that $W$ is $0/1$-valued on
$A\times B$. For almost every $y\in B$ there is a value
$w(M,d(y))\in\{0,1\}$ such that $W(x,y)=w(M,d(y))$ for almost every
$x\in A$. If this value were $1$, then
$d(y)\ge|A|=a>1/2$.
This would put $y$ in $H=A$, a contradiction. Hence $W=0$ almost
everywhere on $A\times B$. Since $W=1$ on $A^2$, we have
$d(x)=\int_A1\dd y=a$ for almost every $x\in A$.
Finally, $B=H^c$ gives $d\le1/2$ there. This proves
\eqref{eq:structural-reduction}.
\end{proof}

\section{Proof of Theorem~\ref{thm:main}}\label{sec:main-proof}

\begin{proof}
Let $W$ be the chosen extremal graphon supplied by
Theorem~\ref{thm:structural-reduction}, and write
$W=K_A\sqcup W_B$, where $a=|A|>1/2$ and $|B|=1-a$. If $a=1$, then
$W=1$ almost everywhere and $\Phi(W)=0\le\pstar$, so the upper bound is
immediate. We may therefore assume $a<1$ for the remainder of the upper-bound
argument and put $x:=(1-a)^{-1}\int_Bd(t)\dd t$;
since $d=a$ on $A$ and $d\le1/2$ on $B$,
\begin{equation}
 \beta=\int_0^1d(t)\dd t=a^2+(1-a)x,
 \qquad 0\le x\le\frac12.
 \label{eq:final-beta}
\end{equation}
Moreover, $\gfun\ge\gcvx$ on $[0,1/2]$, and
Theorem~\ref{thm:jensen} applied on $B$ gives
\begin{equation}
 \int_B\gfun(d(t))\dd t
 \ge\int_B\gcvx(d(t))\dd t
 \ge(1-a)\gcvx(x).
 \label{eq:final-jensen}
\end{equation}
Using \eqref{eq:final-beta} and \eqref{eq:final-jensen} in
\eqref{eq:dirichlet}, and discarding only the nonnegative term
$\Ecal_W(d)$, yields
\begin{align}
 \Phi(W)
 &=\beta(1-\beta)-a\gfun(a)
   -\int_B\gfun(d(t))\dd t-\Ecal_W(d)\notag\\
 &\le\beta(1-\beta)-a\gfun(a)-(1-a)\gcvx(x)
 =F_0(a,x)\le\pstar,
 \label{eq:final-bound}
\end{align}
where the last inequality is \eqref{eq:F0-bound}.

For the reverse inequality, take intervals $A,B$ of measures $a_*$ and
$1-a_*$, put $W=1$ on $A^2$, $W=0$ on $A\times B$, and
$W=\lambda_*$ on $B^2$. Then $d=a_*$ on $A$ and
$d=(1-a_*)\lambda_*=x_*$ on $B$, so
\begin{equation}
 \Phi(W)=F(a_*,x_*)=\pstar.
\label{eq:lower-graphon-value}
\end{equation}
Lemma~\ref{lem:sequence-bridge},
\eqref{eq:lower-graphon-value}, and \eqref{eq:final-bound} therefore give
\begin{equation}
 I(H_3)=\sup_{W\text{ finite weighted}}\Phi(W)
 =\max_{W\text{ graphon}}\Phi(W)=\pstar.
\label{eq:main-value}
\end{equation}

It remains to realize the lower bound in \eqref{eq:main-value} by ordinary
graphs. Let
$s_n=\lfloor a_*n\rfloor$ and $m_n=n-s_n$, and choose an integer
$q_n=x_*n+O(1)$ such that $m_nq_n$ is even; changing
$\lfloor x_*n\rfloor$ by at most one is enough. Since
$0<x_*<1-a_*$, for all sufficiently large $n$ we have
$0\le q_n<m_n$. A simple $q_n$-regular graph can be constructed on
$\mathbb Z/m_n\mathbb Z$ as follows. If $q_n$ is even, join each vertex to
the $q_n/2$ nearest vertices in each cyclic direction. If $q_n$ is odd, then
$m_n$ is even since $m_nq_n$ is even; use the same construction with
$(q_n-1)/2$ neighbors in each direction and add the antipodal perfect
matching. Denote the resulting graph by $R_n$, and set
$G_n=K_{s_n}\sqcup R_n$.
Its red edge density satisfies
\begin{align}
 \beta_{G_n}
 &=\frac{s_n(s_n-1)+m_nq_n}{n(n-1)}
 =a_*^2+(1-a_*)x_*+o(1).
\label{eq:construction-beta}
\end{align}
All clique vertices have normalized degree $(s_n-1)/n=a_*+o(1)$, and all
vertices of $R_n$ have normalized degree $q_n/n=x_*+o(1)$. Every edge has
endpoints of equal normalized degree, so the edge-energy term in
\eqref{eq:finite-dirichlet} is exactly zero. Hence
\eqref{eq:construction-beta} and \eqref{eq:finite-dirichlet} give
\begin{align*}
 p(H_3,G_n)
 &={}\bigl(a_*^2+(1-a_*)x_*\bigr)
       \bigl(1-a_*^2-(1-a_*)x_*\bigr)-a_*\gfun(a_*)-(1-a_*)\gfun(x_*)+o(1)\\
 &=F(a_*,x_*)+o(1)=\pstar+o(1).
\end{align*}
Finally,
$q_n/m_n=x_*/(1-a_*)+o(1)=\lambda_*+o(1)$, which gives the asserted
relative degree inside $R_n$.
\end{proof}

\section{Declaration on the use of AI}

The author used generative AI tools to assist with discussing proof
strategies, proof checking, and exposition. All mathematical arguments,
results, and conclusions were reviewed and verified by the author.

\end{document}